\documentclass[reqno]{amsart}
\usepackage{graphicx}
\usepackage{amscd}
\usepackage{amsmath}
\usepackage{amsfonts}
\usepackage{amssymb}
\usepackage{cite}
\usepackage{xcolor}
\def\R {\mathbb{R}}
\def\C {\mathcal{C}}
\def\G {\mathcal{G}}
\def\c{\mathfrak{C}}
\def\U{\mathcal{U}}
\def\F{\mathcal{F}}
\def\D {\mathbf{D}}
\newtheorem{proposition}{Proposition}[section]
\newtheorem{theorem}[proposition]{Theorem}
\newtheorem{corollary}{Corollary}[section]
\newtheorem{lemma}{Lemma}[section]
\theoremstyle{definition}
\newtheorem{definition}{Definition}[section]
\newtheorem{remark}{Remark}[section]
\numberwithin{equation}{section}
\newtheorem{example}{Example}[section]

\begin{document}

\title[Reconstruction of Parameters of  Subdiffusion]
{Some Unexplored Topics in The Reconstruction\\ of Scalar Parameters
of Subdiffusion}

\author[ S.V. Siryk, L. Tereshchenko, N. Vasylyeva]
{Sergii V. Siryk, Lidiia Tereshchenko and Nataliya Vasylyeva}

\address{CONCEPT Lab, Istituto Italiano di Tecnologia
\newline\indent
Via Morego 30, 16163, Genova, Italy} \email[S.
Siryk]{accandar@gmail.com}

\address{
S.P. Timoshenko Institute of Mechanics of NAS of Ukraine
\newline\indent
Nesterov str.\ 3, 03057 Kyiv, Ukraine} \email[L.
Tereshchenko]{litere70@gmail.com}

\address{Institute of Applied Mathematics and Mechanics of NAS of Ukraine
\newline\indent
G.Batyuka st.\ 19, 84100 Sloviansk, Ukraine;  and
\newline\indent
S.P. Timoshenko Institute of Mechanics of NAS of Ukraine
\newline\indent
Nesterov str.\ 3, 03057 Kyiv, Ukraine}
\email[N.Vasylyeva]{nataliy\underline{\ }v@yahoo.com}


\begin{abstract}
In the paper, we discuss the reconstruction of scalar parameters in
a linear diffusion equation with  fractional in time differential
operators and with additional nonlocal (convolution) terms, which
incorporate memory effects in models. Although, under suitable
assumptions on the data, inverse problems associated with recovery
of these parameters are nowadays well understood, several important
questions related with numerical reconstructions of these parameters
via a nonlocal observation in a small time interval have not yet
been analyzed. This paper aims to provide some answers.
\end{abstract}

\maketitle

\section{Introduction}
\label{si}

\subsection*{Statement of the Model}
Modeling of complex dynamical processes in living media requires
accounting memory (nonlocal) effects. The latter is achieved by
means of introducing convolution integrals and/or  fractional
derivatives in the constitutive  relations, that in turn arrives at
the corresponding evolution equations with convolution terms (see
e.g. \cite{DP,GP}) and/or one- or multi-term fractional differential
operators (FDO) (see for details \cite{Ca,CP,ICM,HPSV,KE,PSV1} and
references therein). One important example of these operators which
will be discussed in this paper is
\begin{equation}\label{i.1}
\mathbf{D}_{t}=\begin{cases}
\sum_{i=1}^{M}\rho_{i}(t)\mathbf{D}_{t}^{\nu_{i}},\qquad\text{the
I type FDO},\\
\\
\sum_{i=1}^{M}\mathbf{D}_{t}^{\nu_{i}}\rho_{i}(t),\qquad\text{the
II type FDO},
\end{cases}
\end{equation}
where integer $M\geq 1$,  $0<\nu_{M}<\ldots<\nu_1<1$, and (variable)
coefficients $\rho_i=\rho_i(t)$ are given.
 The symbol $\mathbf{D}_{t}^{\nu_{i}}$ stands for the
(regularized) left fractional Caputo derivative of order
$\nu_i\in(0,1)$ with respect to time defined as
\[
\mathbf{D}_{t}^{\nu_i}u(\cdot,t)=\begin{cases}
\frac{1}{\Gamma(1-\nu_i)}\frac{\partial}{\partial
t}\int_{0}^{t}\frac{u(\cdot,s)-u(\cdot,0)}{(t-s)^{\nu_i}}ds,\qquad
\nu_{i}\in(0,1),\\
\\
\frac{\partial u}{\partial t}(\cdot,t),\qquad\qquad \nu_i=1,
\end{cases}
\]
where $\Gamma$ is the Euler Gamma-function.

Very often in practice, the parameters of media or models are
unknown or scarcely known, and should be identified by, for example,
solving associated inverse problems for governing differential
equations with fractional  or/and integer derivatives.

Setting for a finite terminal positive time $T$ and a bounded domain
$\Omega\subset\R^{n}$ having smooth boundary
$\partial\Omega\in\C^{2+\alpha},$ $\alpha\in(0,1)$,
\[
\Omega_{T}=\Omega\times (0,T)\qquad
\partial\Omega_{T}=\partial\Omega\times[0,T],
\]
we consider the inverse problems concerning with a reconstruction of
scalar parameters (order of fractional derivatives and of a
singularity in a weakly singular convolution kernel) in the linear
equation with the unknown function $u=u(x,t):\Omega_{T}\to\R,$
\begin{equation}\label{i.2}
\mathbf{D}_{t}\text{\textcolor{red}{$u$}} -\mathcal{L}_{1}u-\mathcal{K}*\mathcal{L}_{2}u=g(x,t),
\end{equation}
supplemented with the initial and the Neumann boundary conditions
\begin{equation}\label{i.3}
\begin{cases}
u(x,0)=u_{0}(x)\qquad\qquad\qquad\qquad\text{in}\qquad \bar{\Omega},\\
\mathcal{N}u+(1-\delta)(\mathcal{K}*\mathcal{N}u)=\varphi(x,t)\quad\text{on}\quad
\partial\Omega_{T}
\end{cases}
\end{equation}
with $\delta=0$ or $1$. The functions $g,u_0,\varphi$ are
prescribed below, $\mathcal{K}$ is a memory integrable kernel.

As for the operator involving in \eqref{i.2} and \eqref{i.3},
$\mathcal{L}_{1}$ and $\mathcal{L}_2$ are linear elliptic
operators of the second order with time- and space-depending
coefficients, while $\mathcal{N}$ is the first order differential
operator, whose detailed description is done in Section \ref{s2}.
Here, as usual, the symbol $"*"$ stands for the time convolution
product, i.e.
\[
h_1* h_2=(h_1* h_2)(t)=\int_{0}^{t}h_1(s)h_2(t-s)ds.
\]

The additional observation in the discussed (below) inverse
problems (IPs) has the form
\begin{equation}\label{i.4}
\int_{\Omega}u(x,t)dx=\psi(t),\quad t\in[0,t^{*}]
\end{equation}
with $t^{*}<\min\{1,T\}$.

In this work, we focus on the two inverse problems: the first of
them concerns with the reconstruction of the orders of leading
derivative and the one of  minor derivatives in $\mathbf{D}_{t}$,
that is $\nu_1$ and $\nu_{i^*}$, $i^{*}\in\{2,3,....,M\}$, while the
second problem deals with the recovery of  $\nu_1$ and the order of
the singularity $\gamma$ in the memory kernel $\mathcal{K}$.

\noindent\textit{Statement of the first inverse problem (FIP):} for
the given right-hand sides in \eqref{i.2}, \eqref{i.3}, coefficients
in the operators $\mathbf{D}_{t},$ $\mathcal{L}_{i},$ $\mathcal{N}$
and the orders $\nu_{i},$ $i\neq i^{*}$ and $1$, and the memory
kernel $\mathcal{K},$ the first inverse problem is the
identification of the triple $(\nu_1,\nu_{i^*},u),$ such that
$\nu_1\in(\nu_2,1),$ $\nu_{i^*}\in(\nu_{i^{*}-1},\nu_{i^{*}+1}),$
and $u$ is a unique (global) classical solution of
\eqref{i.1}-\eqref{i.3} satisfying equality \eqref{i.4} for all
$t\in[0,t^{*}]$.

In order to state the second inverse problem (SIP), we rewrite the
memory kernel $\mathcal{K}$ in the form
\begin{equation}\label{i.5}
\mathcal{K}(t)=t^{-\gamma}\mathcal{K}_{0}(t)
\end{equation}
with unknown order $\gamma\in(0,1)$ and given continuous
$\mathcal{K}_{0}(t)$.

\noindent\textit{Statement of the second inverse problem:} for the
given right-hand sides in \eqref{i.2}, \eqref{i.3}, coefficients in
the operators $\mathbf{D}_{t},$ $\mathcal{L}_{i},$ $\mathcal{N}$ and
the orders $\nu_{i},$ $i\neq 1$, and the memory kernel
$\mathcal{K}_0(t),$ the second inverse problem is the identification
of the triple $(\nu_1,\gamma,u),$ such that
$\nu_1\in(\frac{2\nu_2}{2-\alpha},1),$ $\gamma\in(0,1),$ and $u$ is
a unique (global) classical solution of \eqref{i.1}-\eqref{i.3} with
$\mathcal{K}$ having form \eqref{i.5} and, besides, $u$  satisfies
equality \eqref{i.4} for all $t\in[0,t^{*}]$.

\subsection*{Reconstruction of scalar parameters in \eqref{i.2}: Brief
overview.} The pioneering works concerning a reconstruction of
unknown order of the leading fractional derivative $\nu_1$ in the
one-term I type FDO are \cite{ICM} and \cite{HNWY}, where two
different approaches were proposed. Namely, in \cite{ICM},
minimizing of a certain functional depending on a solution of the
initial-boundary value problem for the equation
\begin{equation}\label{i.6}
\mathbf{D}_{t}^{\nu_1}u-a_1\frac{\partial^{2}u}{\partial x^{2}
}=0,\quad x\in\Omega,\quad t\in[0,T],
\end{equation}
and exploiting additional observation for the flux of $u$ on the one
side of $\partial\Omega$ for $t\in[0,T]$, the authors have
reconstructed $\nu_1$. The further developing of this technique is
carried out in the works \cite{JK1,JKi,SLJ,ZJY}, where $\nu_1$ has
been identified in the case of more general linear autonomous
equation similar to \eqref{i.2} without memory terms (i.e.
$\mathcal{K}\equiv 0$) and of the additional measurement either on
the whole time interval $[0,T]$ or for the terminal time $T$.

The second approach being proposed in \cite{HNWY} deals with the
finding explicit formulas to $\nu_1$ via given measurement. Indeed,
in the aforementioned paper, solving the inverse problem for the
linear equation
\begin{equation}\label{i.7}
\mathbf{D}_{t}^{\nu_1}u-\sum_{i,j=1}^{n}\frac{\partial}{\partial
x_{i}}\Big(a_{ij}\frac{\partial u}{\partial
x_{j}}\Big)-a_0(x)u=0\quad\text{in}\quad \Omega_{T},
\end{equation}
and assuming extra regularity of the initial data,
$u_0\in\C_{0}^{\infty}(\Omega)$, the authors derive the explicit
formulas
\begin{equation}\label{i.8}
\nu_1=\underset{t\to 0}{\lim}\frac{t\frac{\partial u}{\partial
t}(x_0,t)}{u(x_0,t)-u_0(x_0)}\quad\text{and}\quad
\nu_1=\underset{t\to \infty}{\lim}\frac{t\frac{\partial u}{\partial
t}(x_0,t)}{u(x_0,t)}
\end{equation}
with some $x_0\in\Omega$.

It is worth noting that the calculations by  explicit formulas
require only the knowledge of the observation. The latter is a
definitely advantage of the second method due to the first technique
needs not only the measurement but also all information on the
coefficients and the right-hand sides in the corresponding direct
problems. Moreover, an additional shortcoming of the first approach
is huge numerical computations carried out in multidimensional
domains. In spite of aforementioned superiority of the second
technique, the formulas \eqref{i.8} involve the derivative, whose
existence is not guaranteed, in general, by the initial-boundary
value problem for \eqref{i.7}.

In the case of a one-dimensional domain $\Omega$, Janno has removed
this disadvantage for the autonomous linear equation \eqref{i.7}
with the memory term $\mathcal{K}*\frac{\partial^{2}u}{\partial
x^{2}}$ in \cite{J1}, where new explicit formula was obtained
\begin{equation}\label{i.9}
\nu_1=\underset{t\to 0}{\lim}\frac{\ln|u(x_0,t)-u(x_0,0)|}{\ln t}.
\end{equation}
Besides, the author has proposed a numerical approximation of
$\nu_1$ by the pre-limit value of \eqref{i.9} at the smallest
observation time $t=t_a$ (specified by a user)
\begin{equation}\label{i.9*}
\nu_{1,a}\approx\nu_1(t_a)=\frac{\ln|u(x_0,t_a)-u(x_0,0)|}{\ln t_a}.
\end{equation}
Further, in \cite{KPSV1}, the formula \eqref{i.9} has been justified
in the case of semilinear nonautonomous equation
\begin{equation}\label{i.10}
\mathbf{D}_{t}^{\nu_1}u-\mathcal{L}_1
u-\mathcal{K}*\mathcal{L}_2u=f(u)+g(x,t)
\end{equation}
in the multidimensional case of $\Omega$. Moreover, in this work,
exploiting the Tikhonov regularization scheme \cite{TG} and the
multiple quasi-optimality criterion \cite{FNP}, the authors have
proposed (for the first time) the approach to recovery (numerically)
the order $\nu_1$ from the noisy discrete observation data. In
brief, this algorithm consists in the two  steps. At first, assuming
that one has observation $u_k=u(x_0,t_{k})$ at the spatial point
$x_0\in\bar{\Omega}$ and time moments $t_k,$ $k=0,1,...,K,$
$0=t_0<t_1<...<t_{K}\leq t^{*},$ and admitting the presence of a
noise interference $\{\varepsilon_{k}\}_{k=1}^{K}$ deteriorating
those measurements, so that is readily, we have
\[
\Psi_{\varepsilon,k}=u_{k}+\varepsilon_{k},\quad
k=1,2,...,K,\quad\text{and}\quad \Psi_{0,0}=u(x_0,0).
\]
First step of this algorithm is the reconstruction of the function
$\Psi_{\varepsilon}(\sigma,t),$ $t\in[0,t_{K}],$ by means of the
regularized regression from the given noisy data $\Psi_{\varepsilon,
k},$ where the regularization is performed in the finite-dimensional
space:
\[
\text{span}\{t^{\beta_{i}},P_{j}^{\text{\textcolor{red}{(}}
0,-\mathrm{a} \text{\textcolor{red}{)}}},\quad i=1,2,..., I,\quad
j=I+1,...,J\},
\]
according to the Tikhonov scheme with the penalty term
$\sigma\|\cdot\|_{L^{2}_{t^{-\mathrm{a}}}(0,t_{K})}$. Here,
$L^2_{t^{-\mathrm{a}}}(0,t_{K})$ is a weighted space
$L^{2}(0,t_{K})$ with an unbounded weight $t^{-\mathrm{a}},$
$\mathrm{a}\in(0,1);$ $P_{j}^{(0,-\mathrm{a})}$ are Jacobi
polynomial shifted to $[0,t_{K}]$. Then, according to the formula
\eqref{i.9}, the approximate value to $\nu_1$ is computed via
\[
\nu_{1}(\sigma,t)=\frac{\ln|\Psi_{\varepsilon}(\sigma,t)-\Psi_{0,0}|}{\ln
t},
\]
the latter is calculated by the sequence of (regularization)
parameters: $\sigma\in\{\sigma_{p}\},$ $t\in\{\bar{t}_{q}\}$. At
last, the regularized reconstructor
\begin{equation}\label{i.9**}
\nu_{1,reg}=\nu_{1}(\bar{\sigma},\bar{t})
\end{equation}
is selected from the set of approximate quantities
$\nu_{1}(\sigma_p,\bar{t}_{q})$ by utilizing the two-parametric
quasi-optimality criterion \cite{FNP} choosing
$\bar{\sigma}\in\{\sigma_{p}\}$ and $\bar{t}\in\{\bar{t}_{q}\}$.
Numerical examples in \cite{KPSV1} have demonstrated efficiency of
both formula \eqref{i.9} and the regularized algorithm in the case
of a low noise level. In fine, we quote the paper \cite{KPSV2},
where another explicit reconstruction formula for $\nu_1$ in
\eqref{i.10} is derived
\begin{equation}\label{i.11}
\nu_1=\underset{t\to
0}{\lim}\frac{t[u(x_0,t)-u(x_0,0)]}{\int_{0}^{t}[u(x_0,s)-u(x_0,0)]ds}-1,
\end{equation}
and the regularized algorithm similar to \cite[Section 10]{KPSV1} is
exploited to numerical computations. As shown by the numerical tests
in \cite[Section 6.2]{KPSV2}, in the case of high noise level the
new recovering formula for the order $\nu_1$ gives more accurate
outputs than \eqref{i.9} and, hence, takes advantage over
\eqref{i.9}.

Finally, we mention that the stability in reconstruction of $\nu_1$
by the local measurement $u(x_0,t),$ $t\in[0,t^{*}],$
$x_0\in\bar{\Omega},$ and the influence of noisy measurement on the
computation of $\nu_1$ are discussed for subdiffusion equations in
\cite{LHY,LLY,KPSV1,KPSV2,JK1}.

Identification of $\nu_1$ or $\nu_{i}$ in the case of the I type FDO
\eqref{i.1} with nonnegative time-independent coefficients is
analyzed in \cite{JK1,LIY,LY,LLY}, where the corresponding inverse
problems to linear equation \eqref{i.2} with $\mathcal{K}\equiv 0$
and with $\mathcal{L}_1$ having time-independent coefficients are
discussed. In particular, in the case of the additional measurement
done at the endpoint (if $\Omega\subset\R$) or in the case of an
interior observation (if $\Omega\subset\R^{n},\, n\geq 2$), the
unique recovery of multiple orders is established in \cite{LY,LIY}.
To this end, the authors exploited the asymptotic behavior of the
multinomial Mittag-Leffler functions. In \cite{LIY}, Li et al prove
the unique reconstruction of orders $\nu_i$ and several coefficients
from data consisting of an appropriate defined Dirichlet-to-Neumann
map, with a special type Dirichlet boundary excitation.

 The extra
conditions on the coefficients in the operators
$\mathbf{D}_{t},\mathcal{L}_i$ (including sign conditions,
time-independence) are relaxed in \cite{PSV2}, where the authors
have reconstructed $\nu_1$ in the case of semilinear equation like
\eqref{i.2} (i.e. $g=g(x,t,u)$) with time- and space-depending
coefficients in $\mathbf{D}_t$, i.e. $\rho_i=\rho_i(x,t)$. In this
work, the authors analyzed the corresponding inverse problems in the
case of all kinds of boundary conditions including a fractional
dynamic boundary condition. The latter means that unknown order
$\nu_1$ is not only in the equation \eqref{i.2} but also in the
boundary condition.  It is worth noting that, in this work, $\nu_1$
is reconstructed by means of either \eqref{i.9} or the slightly
modified formula, and a key tool in the carried out analysis  is the
one-valued global classical solvability of the corresponding direct
problems to linear and semilinear equations \eqref{i.2} (see for
details \cite{PSV1,SV,V1}). In conclusion, we quote \cite{SV1},
where the justification of the recovery formula \eqref{i.13} has
done for the semilinear subdiffusion equations with the fractional
differential operator~\eqref{i.1}.

Coming to the first and  second inverse problems
\eqref{i.1}-\eqref{i.5}, under certain assumptions on the given
data, the explicit formulas for unknown scalar parameters are
derived in \cite{HPV,HPSV}. Namely, assuming
$\Omega=(\mathfrak{l}_{1},\mathfrak{l}_{2})$ in the one-dimensional
case and setting
\begin{align*}\label{i.12}\notag
\mathcal{I}(t)&=\begin{cases}
\int_{\partial\Omega}\varphi(x,t)dx,\quad\qquad\text{ if}\qquad n\geq 2,\\
\varphi(\mathfrak{l}_{2},t)-\varphi(\mathfrak{l}_{1},t),\qquad
\text{if}\quad n=1,
\end{cases}\\\notag
\c_{\nu}(t)&=\int_{\Omega}g(x,t)dx-\delta(\mathcal{K}*\mathcal{I})(t)+a_{0}(t)\psi(t)+(\mathcal{K}*b_{0}\psi)(t)-\mathcal{I}(t),\quad
\c_{\nu,0}=\c_{\nu}(0),
\end{align*}
and
\begin{equation}\label{i.12}
\mathcal{F}_{\nu}(t)=\begin{cases}
\rho_{i^*}^{-1}(t)[\c_{\nu}(t)-\sum\limits_{i=1,i\neq
i^*}^{M}\rho_{i}(t)\D_{t}^{\nu_{i}}\psi(t)]\qquad
\text{in the case of the I type FDO},\\
\c_{\nu}(t)-\sum\limits_{i=1,i\neq
i^*}^{M}\D_{t}^{\nu_{i}}(\rho_{i}(t)\psi(t))\qquad\qquad \text{ in
the case of the II type FDO},
\end{cases}
\end{equation}
\begin{equation}\label{i.12*}
\mathcal{F}_{\gamma}(t)=\int_{\Omega}g(x,t)dx+a_{0}(t)\psi(t)-\mathbf{D}_t\psi-\mathcal{I}(t),
\end{equation}
where $a_0$ and $b_0$ are the coefficients in $\mathcal{L}_1$ and
$\mathcal{L}_2$ at the  function $u$, respectively, unknown scalar
parameters are computed via formulas
\begin{equation}\label{i.13}
\nu_{1}=
\begin{cases}
\underset{t\to
0}{\lim}\frac{\ln|\psi(t)-\int\limits_{\Omega}u_{0}(x)dx|}{\ln
t}\qquad\qquad\qquad\quad\text{in the case of the I type FDO},
\\
\\
\underset{t\to
0}{\lim}\frac{\ln|\rho_{1}(t)\psi(t)-\rho_{1}(0)\int\limits_{\Omega}u_{0}(x)dx|}{\ln
t}\qquad\quad\text{in the case of the II type FDO},
\end{cases}
\end{equation}
and
\begin{equation}\label{i.14}
\nu_{i^*}=\nu_{1}-\log_{\lambda}\bigg|\underset{t\to
0}{\lim}\frac{\mathcal{F}_{\nu}(\lambda
t)}{\mathcal{F}_{\nu}(t)}\bigg|,
\end{equation}
\begin{equation}\label{i.15}
\gamma=1-\log_{\mu}\bigg|\underset{t\to
0}{\lim}\frac{\mathcal{F}_{\gamma}(\mu
t)}{\mathcal{F}_{\gamma}(t)}\bigg|
\end{equation}
with some $\lambda,\mu\in(0,1).$

Moreover, questions concern to the stability, the uniqueness and the
numerical computations via  \eqref{i.13}-\eqref{i.15} are also
explored in the aforementioned  works. We note that, to compute
$(\nu_1,\nu_{i^*})$ and $(\nu_1,\gamma)$ via formulas
\eqref{i.13}-\eqref{i.15}, the two-step regularized recovery
algorithm is incorporated. The first step of this algorithm is very
similar to the technique of computation of $\nu_1$ in the case of
the one-term $\mathbf{D}_t$ written above. As for the second step
dealing with numerical approximation to $\nu_{i^*}$ or $\gamma$, the
authors substitute the computed $\nu_{1,reg}$ into the corresponding
relations in \eqref{i.12} or \eqref{i.12*} and, then, exploit again
the Tikhonov regularization scheme and the quasi-optimality approach
to calculate $\nu_{i^*}$ or $\gamma$ via \eqref{i.14} or
\eqref{i.15}, respectively.

Therefore, having said that the picture is now pretty clear, there
are still some unexplored issues not discussed so far in the
literature. Namely, this paper is addressed to the questions
concerning numerical calculations of $\nu_1$, $\nu_{i^*}$ and
$\gamma$ via \eqref{i.13}-\eqref{i.15}. Obviously, numerical
computations of unknown parameters $\nu_1,\nu_{i^*},\gamma$ can be
done (similarly to \eqref{i.9*}) via numerical approximation,
$\nu_{1,a},\nu_{i^*,a},\gamma_{a}$, of the searched parameters by
the pre-limit values of \eqref{i.13}-\eqref{i.15} at the smallest
user-defined  time $t=t_a$. On this route, the natural question is
\textit{how should the user choose $t_a$ to ensure a given accuracy
of numerical computations?} In this work, we specify the time
interval $[0,T_{a}]$ (where in general  $T_a$ may be  different for
the case of the first and second inverse problems) such that for any
$t_a\in(0,T_a]$, the values $\nu_{1,a},\nu_{i^*,a},\gamma_{a}$ are
computed with the  predetermined  accuracy. It is worth noting that,
the knowledge of $T_a$ is also important for numerical computations
of the unknown parameters in practice, that is if we use regularized
reconstruction approach. Indeed, in this situation $T_a$ indicates
the time points at which observations are made, stating that at
least $t_1$ should belong to $[0,T_a]$.

Another important issue is that available measurements are typically
distorted by noise (and the corresponding noise levels are often
unknown in practice). Thus, in Section~\ref{s8},
 we discuss a practical algorithm for the approximate reconstruction of the integral measurement
  from discrete noisy data. This algorithm relies on the Tikhonov regularization framework and noise-level-free parameter choice rules.
  It is implemented in a finite-dimensional form as a span of orthogonal Jacobi polynomials and power functions
  (the latter facilitates capturing small-time asymptotics). In general, here we follow the guidelines of
  our previous works (on both multi-term FDO problems \cite{HPV,HPSV,PSV2} as well as significantly
  simpler single-term FDO problems \cite{KPSV1,KPSV2}), except for the implementation of the quasi-optimality principle,
  where a new representation that simultaneously incorporates both recovered scalar parameters is proposed.
  Specifically, the sequential two-stage reconstruction algorithms utilized in our previous works \cite{HPV,HPSV}
   are inherently vulnerable to error propagation; computing the primary order $\nu_1$ and substituting this numerical
   approximation to evaluate limits for the secondary parameters (such as $\nu_{i^*}$ or $\gamma$) can adversely amplify
   noise artifacts in the final outputs. To guarantee higher numerical stability, the present approach abandons sequential
    parameter selection in favor of a simultaneous scheme. Namely, by defining a two-component vector
     $\Vec\nu_\delta \equiv (\nu_{1,\delta},\nu_{i^*,\delta})$ (or $\Vec\nu_\delta \equiv (\nu_{1,\delta},\gamma_{\delta})$)
     and employing a weighted Euclidean norm $\|\Vec\nu_\delta\|^\prime$ to evaluate the quasi-optimality criterion,
      the optimal regularization parameters are selected based on the joint quality of both recovered quantities.
      This simultaneous parameter choice acts as a robuster stabilizer, effectively mitigating the compounding
      uncertainties inherent to sequential limit evaluations and improving the overall resilience of the inverse
      problem solution against observation noise.


\subsection*{Outline of the paper} In the next section, we describe
functional spaces along with notations and state the main
assumptions in the model. The main results concerning identification
$T_a$ in the case of the first and second inverse problems (Theorems
\ref{t2.3}-\ref{t2.5}) are given in Section \ref{s2.3}. We note that
Theorem \ref{t2.3} is related with the time interval $[0,T_{I}]$ for
which the error estimate $|\nu_1-\nu_{1,a}|\leq\varepsilon_{I}$
holds for the predetermined $\varepsilon_{I}\in(0,1)$, while
Theorems \ref{t2.4} and \ref{t2.5} indicate the time intervals where
the error estimates for $\nu_{i*}$ and $\gamma$ hold. The proofs of
Theorems \ref{t2.3}-\ref{t2.5} are carried out in Sections
\ref{s4}-\ref{s6}. Section \ref{s3} is devoted to auxiliary
technical results playing a key role in the verification of Theorems
\ref{t2.3}-\ref{t2.5}. In Section \ref{s7}, we discuss some
sufficient conditions on the given data in \eqref{i.1}-\eqref{i.4}
that ensure assumption h7 required in Theorems \ref{t2.4} and
\ref{t2.5}, which in general is implicit requirement on the
right-hand sides of \eqref{i.2}-\eqref{i.3} and on the measurement
$\psi$. In conclusion, the computational approach for simultaneous
regularized reconstruction $(\nu_1,\nu_{i^*})$ and $(\nu_1,\gamma)$
alongside with numerical tests are discussed in Section~\ref{s8}.

\section{Notations and General Hypothesis  in the Model}
\label{s2}

\subsection{Functional setting}\label{s2.1*}
In our analysis, for any nonnegative integer $l$ and for any $p\geq
1,$ $\alpha\in(0,1),$ and any Banach space $(X,\|\cdot\|_{X}),$ we
will use the usual spaces
\[
\C([0,T],X),\quad \C^{l+\alpha}(\bar{\Omega}),\quad
\C^{l+\alpha}([0,T]),\quad L^{p}(0,T),
\]
and for $l=0,1,2,$ and $\nu\in(0,1),$ we will exploit the fractional
H\"{o}lder spaces
\[
\C_{\nu}^{\alpha}([0,T])\qquad
\C^{l+\alpha,\frac{l+\alpha}{2}\nu}(\bar{\Omega}_{T}).
\]
The detail properties of $\C^{l+\alpha,\frac{l+\alpha}{2}\nu}$,
which first was introduced in \cite{KV}, are described in
\cite[Section 2]{KPV1}, here below we  only recall the definition of
these fractional spaces.
\begin{definition}\label{d2.1}
The functions $v=v(x,t)$ and $w=w(t)$ belong to the classes
$\C^{l+\alpha,\frac{l+\alpha}{2}\nu}(\bar{\Omega}_{T})$, $l=0,1,2,$
 and $\C_{\nu}^{\alpha}([0,T]),$ respectively, if $v$ and $w$
 together with their corresponding derivatives are continuous on
 $\bar{\Omega}_{T}$ and $[0,T],$ respectively, and the norms below
 are finite
\begin{align*}
\|v\|_{\C^{l+\alpha,\frac{l+\alpha}{2}\nu}(\bar{\Omega}_{T})}&=
\begin{cases}
\|v\|_{\C([0,T],\C^{l+\alpha}(\bar{\Omega}))}+\sum\limits_{|j|=0}^{l}\langle
D_{x}^{j}v\rangle_{t,\Omega_{T}}^{(\frac{l+\alpha-|j|}{2}\nu)},\qquad\qquad\qquad\qquad\qquad
l=0,1,\\
\|v\|_{\C([0,T],\C^{2+\alpha}(\bar{\Omega}))}+\|\D_{t}^{\nu}v\|_{\C^{\alpha,\frac{\nu\alpha}{2}}(\bar{\Omega}_{T})}+\sum\limits_{|j|=1}^{2}\langle
D_{x}^{j}v\rangle_{t,\Omega_{T}}^{(\frac{2+\alpha-|j|}{2}\nu)},\qquad
l=2,
\end{cases}
\\
\|w\|_{\C_{\nu}^{\alpha}([0,T])}&=\|w\|_{\C([0,T])}+\|\mathbf{D}_{t}^{\nu}w\|_{\C([0,T])}+\langle\mathbf{D}_{t}^{\nu}w\rangle_{t,[0,T]}^{(\alpha)},
\end{align*}
where $\langle \cdot\rangle_{t,\Omega_{T}}^{(\alpha)}$, $\langle
\cdot\rangle_{t,[0,T]}^{(\alpha)}$ and $\langle
\cdot\rangle_{x,\Omega_{T}}^{(\alpha)}$ stand for the standard
H\"{o}lder seminorms of the function  with respect to time and space
variables, respectively.
\end{definition}
\noindent In a similar way, for $l=0,1,2,$ the spaces
$\C^{l+\alpha,\frac{l+\alpha}{2}\nu}(\partial\Omega_{T})$ is
defined.

In this work, we will also use the Hilbert space
$L_{\varrho}^{2}(t_1,t_2)$ of real-valued square integrable
functions with a positive weight $\varrho=\varrho(t)$ on
$(t_1,t_2).$ The inner product and the norm in
$L_{\varrho}^{2}(t_1,t_2)$ are defined as
\[
\langle
u,v\rangle_{L_{\varrho}^{2}}=\int_{t_1}^{t_2}\varrho(t)v(t)u(t)dt\qquad
\text{and}\quad\|v\|^{2}_{L_{\varrho}^{2}}=\langle
v,v\rangle_{L_{\varrho}^{2}}.
\]

Throughout this paper, we also use  notation $x^*$ for the minimum
point of Gamma function if $x\geq 0$, that is
\[
\Gamma(1+x^*)=\underset{x\geq 0}{\min}\, \Gamma(x),\qquad
x^{*}\approx 0.4616.
\]

\subsection{General assumptions in the model}\label{s2.2}
\begin{description}
\item[h1. Conditions on the operators]
The operators $\mathcal{L}_{i}$ and $\mathcal{N}$  are defined as
\begin{align*}
\mathcal{L}_{1}& = \sum_{ij=1}^{n}\frac{\partial}{\partial
x_{i}}\Big( b_{ij}(x,t)\frac{\partial}{\partial x_{j}}\Big)
+a_{0}(t),\quad \mathcal{L}_{2} =
\sum_{ij=1}^{n}\frac{\partial}{\partial x_{i}}\Big(
b_{ij}(x,t)\frac{\partial}{\partial x_{j}}\Big) +b_{0}(t),\\
\mathcal{N}&=
-\sum_{ij=1}^{n}b_{ij}(x,t)N_{i}\frac{\partial}{\partial x_{j}}
\end{align*}
with $\mathbf{N}=\{N_{1},...,N_{n}\}$ being the unit outward normal
vector to $\Omega$.

\noindent We require that there exist positive constants
$\varrho_{2}>\varrho_1>0,$ such that
     \begin{equation*}
\varrho_{1}|\xi|^{2}
\leq\sum_{ij=1}^{n}b_{ij}(x,t)\xi_{i}\xi_{j}\leq\varrho_{2}|\xi|^{2}
    \end{equation*}
   for any $(x,t,\xi)\in\bar{\Omega}_{T}\times \mathbb{R}^{n}$.
    \item[h2. Conditions on the FDO in \eqref{i.1}] We assume that
\begin{equation*}
\nu_{1}\in\Big(\frac{2\nu_2}{2-\alpha},1\Big)\quad\text{and}\quad
 0<\nu_M<...<\nu_2<1,
\end{equation*}
and, besides, there is  a positive constant $\varrho_{3}$ such that
 \[
\rho_{1}(t)\geq \varrho_{3}>0
 \]
for all $t\in[0,T]$.
    \item[h3. Smoothness of   the coefficients in \eqref{i.1} and \eqref{i.2}] For
    $i,j=1,\ldots,n,$ $k=1,\ldots,M,$ and $\alpha\in(0,1),$ $\nu\in(1,1+\alpha/2),$
    there hold
    \begin{equation*}
  a_{0},
  b_{0}\in
    \C^{\frac{\alpha}{2}}([0,T]),\quad
a_{ij}(x,t)\in
 \C^{1+\alpha,\frac{1+\alpha}{2}}(\bar{\Omega}_{T}),\quad
 \rho_{k}\in\C^{\nu}([0,T]).
    \end{equation*}
             \item[h4. Regularity of the given functions]
       \[
       \mathcal{K}\in L^{1}(0,T),\quad
 \varphi\in\C^{1+\alpha,\frac{1+\alpha}{2}}(\partial\Omega_{T}),\quad
u_{0}\in C^{2+\alpha}(\bar{\Omega}), \quad
g\in\C^{\alpha,\frac{\alpha}{2}}(\bar{\Omega}_{T}).
    \]
    \item[h5. Condition on the additional measurement]
    We require that $\psi\in\C([0,t^{*}])$ has $M$-fractional (left)
    regularized Caputo derivatives of order $\nu_i,$ $i=1,..,M$, and all
    these derivatives are H\"{o}lder continuous.
    \item[h6. Compatibility conditions] For every
         $x\in\partial\Omega$ at the initial time
     $t=0$, there holds
       \begin{equation*}
\mathcal{N}u_{0}(x)|_{t=0}=\varphi(x,0).
    \end{equation*}
    \end{description}

 Throughout this paper, bering in mind assumption h4, we use the notation
\begin{equation}\label{2.0}
\mathcal{R}:=\mathcal{R}(g,u_0,\varphi)=\|g\|_{\C^{\alpha,\alpha/2}(\bar{\Omega}_{T})}+\|\varphi\|_{\C^{1+\alpha,\frac{1+\alpha}{2}}(\partial\Omega_{T})}+
\|u_0\|_{\C^{2+\alpha}(\bar{\Omega})}
\end{equation}

\section{Statement of the Main Results}\label{s2.3}

\noindent First, for reader's convenience, we recall the results
concerning with the solvability of the first and second inverse
problems, which are established in \cite{HPV,HPSV} and are reported
here below in forms (and notations) tailored for our goals. The
first claim subsumes Theorem 2.2 and Lemma 3.4 in \cite{HPSV}.
\begin{theorem}\label{t2.1}
Let positive  $T$ be arbitrary but finite,  and let assumptions
h1-h6 hold. If $\c_{\nu,0}\neq 0$ and $\rho_{i^*}(t)\neq 0$ for all
$t\in[0,t^{*}]$, then the first inverse problem
\eqref{i.1}-\eqref{i.4} has a unique solution
$(\nu_{1},\nu_{i^*},u)$ with $\nu_{1}$ and $\nu_{i^*}$ being
computed via formulas \eqref{i.13} and \eqref{i.14}, respectively,
and with $u$ being a unique global classical solution of the direct
problem \eqref{i.1}-\eqref{i.3} and satisfying \eqref{i.4} for all
$t\in[0,t^*]$. Besides,
\begin{equation}\label{2.1}
\|u\|_{\C^{2+\alpha,\frac{2+\alpha}{2}\nu_1}(\bar{\Omega}_T)}
+\sum_{i=2}^{M}\|\mathbf{D}_{t}^{\nu_{i}}u\|_{\C^{\alpha,\frac{\alpha\nu_1}{2}}(\bar{\Omega}_T)}\leq
C_0\mathcal{R}(g,u_0,\varphi),
\end{equation}
where the positive quantity $C_0$ depends only on $T$, the Lebesgue
measure of $\Omega$, $\|\mathcal{K}\|_{L^1(0,T)}$ and the
corresponding norms of the coefficients in the operators.

Moreover, for any $T_0\in(0,T]$ and each $t\in[0,T_0],$ the
following holds
\begin{equation}\label{2.2}
\int_{\Omega}u(x,t)dx\in\C^{\frac{\alpha\nu_{1}}{2}}_{\nu_1}([0,T_0]),\,
\D_{t}^{\nu_i}\int_{\Omega}u(x,t)dx\in\C^{\frac{\alpha\nu_{1}}{2}}([0,T_0]),\,
i=2,...,M,\,\mathbf{D}_t\int_{\Omega}u(x,t)dx=\c_{\nu}(t).
\end{equation}
In particular, if $T_0=t^*$, then
\begin{equation}\label{2.2*}
\psi\in\C^{\frac{\alpha\nu_{1}}{2}}_{\nu_1}([0,t^*]),\,
\D_{t}^{\nu_i}\psi\in\C^{\frac{\alpha\nu_{1}}{2}}([0,t^*]),\,
i=2,...,M,\,\mathbf{D}_t\psi=\c_{\nu}(t)
\end{equation}
for all $t\in[0,t^*].$
\end{theorem}
In the case of the second inverse problem \eqref{i.1}-\eqref{i.5},
setting
\begin{equation}\label{c.1}
\c_1:=\c_1(t)=\delta\mathcal{I}(t)-b_0(t)\psi(t),\qquad
\c_{1,0}=\c_{1}(0),
\end{equation}
and subsuming Theorem 2.3 and Lemma 4.1 \cite{HPV}, we end up with
the claim.
\begin{theorem}\label{t2.2}
Let arbitrary $T>$ be finite and $\c_{\nu,0},\c_{1,0}\neq 0$. We
also assume that $\mathcal{K}(t)$ has the form \eqref{i.5} with
$\mathcal{K}_{0}(0)\neq 0$ and $\mathcal{K}_{0}\in
L^{1}(0,T)\cap\C([0,t^{*}])$. Under assumptions h1-h6, the second
inverse problem \eqref{i.1}-\eqref{i.5} admits a unique solution
$(\nu_1,\gamma,u)$ with $\nu_1$ and $\gamma$ being calculated via
formulas \eqref{i.13} and \eqref{i.15}, and with
$u\in\C^{2+\alpha,\frac{2+\alpha}{2}\nu_1}(\bar{\Omega}_{T})$ being
a unique global classical solution of the direct problem
\eqref{i.1}-\eqref{i.3} and satisfying \eqref{i.4} for any
$t\in[0,t^*]$. Besides, relations \eqref{2.1}-\eqref{2.2*} hold.
\end{theorem}
Recasting the arguments leading to \cite[Remark 6.4]{HPSV} and
exploiting \cite[Lemmas 5.5-5.6]{SV} along with
\eqref{2.1}-\eqref{2.2*} arrive at the following results.
\begin{corollary}\label{c2.1}
Let assumptions of either Theorem \ref{t2.1} or Theorem \ref{t2.2}
hold. Then, for any $i=1,2,...,M$ and any
$\alpha_0\in(0,\frac{\alpha\nu_1}{2}]$ there are the following
estimates:
\begin{align*}
\|\mathbf{D}_{t}^{\nu_i}\psi\|_{\C^{\alpha_0}([0,t^*])}&\leq
C_0|\Omega|\mathcal{R}(g,u_0,\varphi),\\
\langle\mathbf{D}_{t}^{\nu_i}(\rho_i\psi)\rangle_{t,[0,t^*]}^{(\alpha_0)}&\leq
C_1\|\rho_i\|_{\C^{\nu}([0,t^*])}\|\mathbf{D}_{t}^{\nu_i}\psi\|_{\C^{\alpha_0}([0,t^*])}
\leq
C_1C_0|\Omega|\|\rho_i\|_{\C^{\nu}([0,t^*])}\mathcal{R}(g,u_0,\varphi),\\
\|\mathbf{D}_{t}^{\nu_i}(\rho_i\psi)\|_{\C^{\alpha_0}([0,t^*])}&\leq
C_2\|\rho_i\|_{\C^{\nu}([0,t^*])}\|\mathbf{D}_{t}^{\nu_i}\psi\|_{\C^{\alpha_0}([0,t^*])}
\leq
C_2C_0|\Omega|\|\rho_i\|_{\C^{\nu}([0,t^*])}\mathcal{R}(g,u_0,\varphi),\\
\|\mathbf{D}_{t}\psi\|_{\C^{\alpha_0}([0,t^*])}&=\|\c_\nu\|_{\C^{\alpha_0}([0,t^*])}\leq
C_3\mathcal{R}(g,u_0,\varphi),
\end{align*}
where the positive quantities $C_1$ and $C_2$ depend only on $t^{*}$
and $\nu$, while
\[
C_3=C_0|\Omega|\max\{1,C_2\}\sum_{i=1}^{M}\|\rho_i\|_{\C^{\nu}([0,t^*])}.
\]
\end{corollary}

At this point, setting an arbitrary time-point $t_a\in(0,t^*]$, we
define the  approximate values $\nu_{1,a}$, $\nu_{i^*,a}$ and
$\gamma_a$ for $\nu_{1}$, $\nu_{i^*}$ and $\gamma$ calculated via
\eqref{i.13}-\eqref{i.15}, respectively, as
\begin{equation}\label{2.3}
\nu_{1,a}=
\begin{cases}
\frac{\ln|\psi(t_a)-\int_{\Omega}u_{0}(x)dx|}{\ln
t_a}\qquad\qquad\qquad\quad\text{for the I type FDO},
\\
\frac{\ln|\rho_{1}(t_a)\psi(t_a)-\rho_{1}(0)\int_{\Omega}u_{0}(x)dx|}{\ln
t_a}\qquad\quad\text{for the II type FDO},
\end{cases}
\end{equation}
and
\begin{equation}\label{2.4}
\nu_{i^*,a}=\nu_{1,a}-\log_{\lambda}\bigg|\frac{\mathcal{F}_{\nu,a}(\lambda
t_a)}{\mathcal{F}_{\nu,a}(t_a)}\bigg|,
\end{equation}
\begin{equation}\label{2.5}
\gamma_a=1-\log_{\mu}\bigg|\frac{\mathcal{F}_{\gamma,a}(\mu
t_a)}{\mathcal{F}_{\gamma,a}(t_a)}\bigg|,
\end{equation}
where we put
\begin{equation}\label{2.4*}
\mathcal{F}_{\nu,a}(t_a)=\begin{cases}
\frac{1}{\rho_{i^*}(t_a)}[\c_{\nu}(t_a)-\sum\limits_{i=2,i\neq
i^*}^{M}\rho_{i}(t_a)\D_{t}^{\nu_{i}}\psi(t_a)-\rho_{1}(t_a)\D_{t}^{\nu_{1,a}}\psi(t_a)]\quad
\text{for the I type FDO},\\
\c_{\nu}(t_a)-\sum\limits_{i=2,i\neq
i^*}^{M}\D_{t}^{\nu_{i}}(\rho_{i}(t_a)\psi(t_a))-\D_{t}^{\nu_{1,a}}(\rho_{1}(t_a)\psi(t_a))\quad
\text{for the II type FDO},
\end{cases}
\end{equation}
\begin{equation}\label{2.5*}
\mathcal{F}_{\gamma,a}(t_a)=\begin{cases}
\int_{\Omega}g(x,t_a)dx+a_{0}(t_a)\psi(t_a)-\mathcal{I}(t_a)\\
-\sum\limits_{i=2,i\neq
i^*}^{M}\rho_{i}(t_a)\D_{t}^{\nu_{i}}\psi(t_a)-\rho_{1}(t_a)\D_{t}^{\nu_{1,a}}\psi(t_a)\qquad
\text{for the I type FDO},\\
\\
\int_{\Omega}g(x,t_a)dx+a_{0}(t_a)\psi(t_a)-\mathcal{I}(t_a)\\
-\sum\limits_{i=2,i\neq
i^*}^{M}\D_{t}^{\nu_{i}}(\rho_{i}(t_a)\psi(t_a))-\D_{t}^{\nu_{1,a}}(\rho_{1}(t_a)\psi(t_a))\quad
\text{for the II type FDO},
\end{cases}
\end{equation}
Clearly, the approximate formulas \eqref{2.3}-\eqref{2.5} tell us
that values $\nu_{1,a},$ $\nu_{i^*,a}$ and $\gamma_{a}$ depend on
$t_a$. We recall that, in this paper, we aim to find time intervals
for which these quantities are computed with the given accuracy
$\varepsilon_{I},$ $\varepsilon_{II}$ and
$\varepsilon_{III}\in(0,1)$, respectively. Other words, we find or,
at least, evaluate the terminal time $T_{I},$ $T_{II}$ and $T_{III}$
such that the following error estimates hold
\begin{align}
\label{2.6.1}
\Delta_1(t_a)&:=|\nu_1-\nu_{1,a}|<\varepsilon_{I}\qquad \text{for
any}\quad t_a\in(0,T_{I}],\\\label{2.6.2}
\Delta_2(t_a)&:=|\nu_{i^*}-\nu_{i^*,a}|<\varepsilon_{II}\quad
\text{for any}\quad t_a\in(0,T_{II}],\\\label{2.6.3}
\Delta_3(t_a)&:=|\gamma-\gamma_{a}|<\varepsilon_{III}\qquad
\text{for any}\quad t_a\in(0,T_{III}].
\end{align}
It is apparent that the values $T_{I},$ $T_{II}$ and $T_{III}$ may
depend on $\varepsilon_{I},$ $\varepsilon_{II}$ and
$\varepsilon_{III}$, correspondingly.

 Our first result concerns to
the bound \eqref{2.6.1} and, accordingly, $T_{I}$ in the case of
both inverse problems.

Setting
\begin{equation}\label{2.10*}
\bullet\quad T_{I}^0=\begin{cases}
\min\Big\{t^*,\Big(\frac{\rho_1(0)}{\Gamma(1+x^*)|\c_{\nu,0}|}\Big)^{-\frac{2}{\varepsilon_{I}}},
\Big(\frac{|\c_{\nu,0}|}{\Gamma(1+x^*)\rho_1(0)}\Big)^{-\frac{2}{\varepsilon_{I}}},(1-\varepsilon_{I})^{\frac{2}{\varepsilon_{I}}}
\Big\}\quad \text{for the I type FDO},\\
\\
\min\Big\{t^*,(\Gamma(1+x^*)|\c_{\nu,0}|)^{\frac{2}{\varepsilon_{I}}},
\Big(\frac{|\c_{\nu,0}|}{\Gamma(1+x^*)}\Big)^{-\frac{2}{\varepsilon_{I}}},(1-\varepsilon_{I})^{\frac{2}{\varepsilon_{I}}}
\Big\}\qquad \text{for the II type FDO},
\end{cases}
\end{equation}

\noindent$\qquad\bullet$ $T_{K}$ is the maximum positive time in
$[0,t^*]$ such that
\begin{equation}\label{2.10}
sgn\,(\mathcal{K}_{0}(t))= sgn\,(\mathcal{K}_{0}(0)),
\end{equation}
\begin{equation}\label{2.10**}
\bullet\quad C_{4}=\frac{C_0}{\Gamma(1+x^*)}\cdot \begin{cases}
1+2\sum\limits_{i=2}^{M}\frac{|\rho_i(0)|}{\rho_1(0)\Gamma(1+x^*)} \qquad\qquad\qquad\qquad \text{for the I type FDO},\\
[C_1+C_2]\Big[1+\frac{2}{\Gamma(1+x^*)}\Big]\sum\limits_{i=1}^{M}\|\rho_i\|_{\C^{\nu}([0,t^*])}
\quad \text{for the II type FDO},
\end{cases}
\end{equation}
we claim the following.
\begin{theorem}\label{t2.3}
For any $\varepsilon_{I}\in(0,1),$ under assumptions of either
Theorem \ref{t2.1} or Theorem \ref{t2.2}, the bound \eqref{2.6.1}
holds for each $t_a\in(0,T_{I}]$ with some $T_{I}$ satisfying the
inequality
\begin{equation*}\label{2.7}
T_{I}\leq T_{I}^{0}.
\end{equation*}
If in the first inverse problem, the FDO has at least three terms,
i.e. $M\geq 3$, then
\begin{equation}\label{2.8}
T_{I}=\begin{cases}
\min\Big\{T_{I}^{0},\Big(\frac{\varepsilon_I|\c_{\nu,0}|}{C_4\rho_1(0)\mathcal{R}(g,u_0,\varphi)}\Big)^{1/\nu_0}\Big\} \quad \text{for the I type FDO},\\
\min\Big\{T_{I}^{0},\Big(\frac{\varepsilon_I|\c_{\nu,0}|}{C_4\mathcal{R}(g,u_0,\varphi)}\Big)^{1/\nu_0}\Big\}
\quad\quad \text{ for the II type FDO},
\end{cases}
\end{equation}
with
\[
\nu_0=\begin{cases} \frac{\alpha\nu_2}{2},\qquad \text{if}\qquad i^*\neq 2,\\
\frac{\alpha\nu_3}{2},\qquad\text{if}\qquad i^*=2.
\end{cases}
\]
If in the second inverse problem, unknown $\gamma$ is looking for in
the interval $(0,\gamma_0)$ with given $\gamma_0<1$, then
\begin{equation}\label{2.9}
T_{I}=\begin{cases}
\min\Big\{T_{I}^{0},\Big(\frac{\varepsilon_I|\c_{\nu,0}|}{C_4\rho_1(0)\mathcal{R}(g,u_0,\varphi)}\Big)^{2/\alpha\nu_2},T_{K}\Big\} \quad \text{for the I type FDO},\\
\min\Big\{T_{I}^{0},\Big(\frac{\varepsilon_I|\c_{\nu,0}|}{C_4\mathcal{R}(g,u_0,\varphi)}\Big)^{2/\alpha\nu_2},T_K\Big\}
\quad\quad \text{ for the II type FDO},
\end{cases}
\end{equation}
where $\mathcal{R}(g,u_0,\varphi)$ is defined via \eqref{2.0}.
\end{theorem}
\begin{remark}\label{r2.1}
It is apparent that the existence of the positive $T_{K}$ (see
\eqref{2.10}) is ensured by the continuity of the kernel
$\mathcal{K}_0$ on $[0,t^*]$ and nonvanishing of $\mathcal{K}_0(0)$
prescribed by the requirements of Theorem \ref{t2.2}.
\end{remark}
\begin{remark}\label{r2.2}
Obviously, $T_I$ defined either via \eqref{2.8} or via \eqref{2.9}
depends on the value $C_4$ which in turn does on the quantities
$C_0,$ $C_1$ and $C_2$ (see \eqref{2.10**}). Generally speaking, the
(explicitly) computation of these values is a complex task.
Moreover, we note that $C_0$ (as stated in Theorem \ref{t2.1})
depends on $\|\mathcal{K}\|_{L^{1}(0,T)}$ that brings additionally
difficulties to computation in the case of the second inverse
problem. However, as it will be demonstrated in the proof of Theorem
\ref{t2.3} (see \eqref{4.3}), the last drawback will be removed
thanks to introducing time $T_K$ and the restriction on $\gamma$.
\end{remark}

In further analysis, we focus only on the case $\nu_{1,a}<1$ and
Theorem \ref{t2.3} tells us that this situation may be achieved by a
restriction on the error $\varepsilon_{I}$ and, accordingly, on the
time $T_{I}$.

Namely, thanks to  $\nu_{i}\in(0,1),$ $i=1,...,M,$ we may assume the
existence of two positive values $\underline{\nu}$ and $\bar{\nu}$
such that
\begin{equation}\label{2.18}
0<\underline{\nu}<\nu_{M}<...<\nu_2<\nu_1<\bar{\nu}<1.
\end{equation}
Then, performing straightforward calculations, we claim
\begin{corollary}\label{c.f1}
If $\varepsilon_{I}\in(0,1-\bar{\nu}),$ then $\nu_{1,a}$ computed by
\eqref{2.3} satisfies to one of the following inequalities:
\[
\text{either}\quad \nu_{1,a}<\nu_1\quad\text{or}\quad
\nu_{1,a}<\bar{\nu}+\varepsilon_{I}
\]
for each $t_a\in(0,T_{I}]$ with $T_{I}$ is given in Theorem
\ref{t2.3}.
\end{corollary}

Our next results concerns with the estimate \eqref{2.6.2} and some
properties of the functions in formulas \eqref{2.4} and
\eqref{i.14}.

Denote by
\begin{equation}\label{2.13}
\c_2:=\c_2(t)=\begin{cases}
\rho_1[\mathbf{D}_{t}^{\nu_1}-\mathbf{D}_{t}^{\nu_{1,a}}]\psi\qquad\,\text{
for
the I type FDO,}\\
[\mathbf{D}_{t}^{\nu_1}-\mathbf{D}_{t}^{\nu_{1,a}}](\rho_1\psi)\qquad\text{for
the II type FDO},
\end{cases}
\end{equation}
and for each $n\in\mathbb{N}$
\begin{equation}\label{2.14}
\mathcal{U}(t,n)=
\begin{cases}
\frac{\mathbf{D}_{t}^{\nu_1}\psi(t)}{n}+\frac{1}{\rho_{i^*}(t)}[\c_{\nu}(t)-\sum\limits_{i\neq
i^*}\rho_{i}(t)\mathbf{D}_{t}^{\nu_i}\psi(t)] \quad\text{ for
the I type FDO,}\\
\frac{\mathbf{D}_{t}^{\nu_1}(\rho_{i^*}(t)\psi(t))}{n}+\c_{\nu}(t)-\sum\limits_{i\neq
i^*}\mathbf{D}_{t}^{\nu_i}(\rho_{i}(t)\psi(t)) \quad\text{for
the II type FDO.}\\
\end{cases}
\end{equation}
It is worth noting that the arguments of \cite[Section 3]{HPSV}
establish the existence $n^*\in\mathbb{N}$ such that
\begin{equation}\label{2.16}
\mathcal{U}(0,n^*)\neq 0.
\end{equation}
Further, making additional requirement on the observation:
\begin{description}
 \item[h7] There exists a positive time $t_1^*\in(0,t^*]$ such that
 $\nu_{1,a}<1$ and
 $\mathbf{D}_{t}^{\nu_{1,a}}\psi$ belongs to
 $\C^{\alpha_1}([0,t_1^*])$ for some $\alpha_1\in(0,1),$ and
 \begin{equation}\label{2.15}
\c_{2}(0)\neq 0,
 \end{equation}
    \end{description}
    we claim the following.
\begin{lemma}\label{l2.1}
Let assumptions of Theorem \ref{t2.1} and h7 hold. Then there are
the following:

\noindent (i)
$\|\mathbf{D}_{t}^{\nu_{1,a}}(\psi\rho_1)\|_{\C^{\alpha_1}([0,t^*_{1}])}\leq
C_5\|\rho_1\|_{\C^{\nu}([0,t^*])}
\|\mathbf{D}_{t}^{\nu_{1,a}}\psi\|_{\C^{\alpha_1}([0,t^*_{1}])},$

\noindent with the positive value $C_5$ depending only on $t^*$ and
$\nu$;

\noindent (ii) for any
$\alpha_2\in(0,\min\{\alpha_1,\alpha\nu_1/2\}]$, there is the
estimate
\[
\|\c_2\|_{\C^{\alpha_2}([0,t_1^*])}\leq
C_6\mathcal{R}_1(\rho_1,\mathcal{R},\mathbf{D}_{t}^{\nu_{1,a}}\psi)
\]
with  $C_6=\max\{C_0|\Omega|,C_2C_0|\Omega|,C_5\}$ and
$\mathcal{R}_1=\|\rho_1\|_{\C^{\nu}([0,t^*])}[\mathcal{R}+\|\mathbf{D}_{t}^{\nu_{1,a}}\psi\|_{\C^{\alpha_1}([0,t^*_{1}])}]$;

\noindent (iii) for any $\alpha_0\in(0,\alpha\nu_1/2],$
$\alpha_2\in(0,\min\{\alpha_1,\alpha\nu_1/2\}]$ and each
$n\in\mathbb{N}$, there hold
\begin{align*}
\|\mathcal{F}_{\nu}\|_{\C^{\alpha_0}([0,t^*])}&\leq
C_7\mathcal{R}(g,u_0,\varphi),\qquad
\|\mathcal{F}_{\nu,a}\|_{\C^{\alpha_2}([0,t_{1}^*])}\leq
C_{8}[\mathcal{R}+\mathcal{R}_1],\\
\|\mathcal{U}\|_{\C^{\alpha_0}([0,t^*])}&\leq
[C_{7}+C_0|\Omega|\max\{1,C_2\|\rho^{-1}_{i^*}\|_{\C^{\nu}([0,t^*])}\}]\mathcal{R}(g,u_0,\varphi),
\end{align*}
where
$C_7=\max\{1,\|\rho^{-1}_{i^*}\|_{\C^{\nu}([0,t^*])}\}[C_0|\Omega|\max\{1,C_2\}\sum\limits_{i\neq
i^*}\|\rho_i\|_{\C^{\nu}([0,t^*])}+C_3]$, $C_8=C_6+C_7$.
\end{lemma}
\noindent Indeed, the estimates of this lemma are easily verified by
straightforward calculations using Theorem \ref{t2.1} (see
\eqref{2.1}-\eqref{2.2*}) and Corollary \ref{c2.1}.

At this point, introducing the value
\[ \varepsilon_\nu=\begin{cases}
\nu_{1,a}-\nu_{i^*+1},\quad
\text{if}\quad i^*\neq M,\\
\nu_{1,a}-\underline{\nu},\qquad\quad \text{if}\quad i^*=M,
\end{cases}
\]
 we state our main result concerning with the estimate \eqref{2.6.2}.
\begin{theorem}\label{t2.4}
Let assumptions of Theorem \ref{t2.1} and h7 hold. Then for any
$\lambda\in(0,1),$ and $\varepsilon_{II}\in(\varepsilon_{\nu},1),$
and each
$\varepsilon\in(0,1-\lambda^{\frac{\varepsilon_{II}-\varepsilon_{\nu}}{3}})$,
the inequality \eqref{2.6.2} is fulfilled for each
$t_a\in(0,T_{II}]$ with some positive $T_{II}$ satisfying the
inequality
\[
T_{II}\leq \min\{t_{1}^*,T_{I}^{0}\},
\]
where $T_{I}^0$  is defined via \eqref{2.10*} with
\begin{equation}\label{e.e}
\varepsilon_{I}\in\begin{cases} (0,
\min\{\nu_{i^*+1},1-\bar{\nu}\}),\quad\text{if}\quad i^*\neq M,\\
(0, \min\{\underline{\nu},1-\bar{\nu}\}),\qquad\quad\text{if}\quad
i^*= M.
\end{cases}
\end{equation}

\noindent In the case of $M\geq 3$, the value $T_{II}$ is defined as
\[
T_{II}=\min\Big\{t_{1}^*,T_{I},
(2n^*)^{-1/\nu_0},\Big[\frac{C_9\varepsilon}{1+n^*C_9\varepsilon}\Big]^{1/\nu_0},
\Big[\frac{\varepsilon|\c_2(0)|}{3C_8(\mathcal{R}+\mathcal{R}_1)}\Big]^{1/\alpha_3}\Big\}
\]
with $\alpha_3=\min\{\alpha_1,\nu_0\}$ and
$$
C_9=\frac{\Gamma(1+x^*)|\mathcal{U}(0,n^*)|}{3\Gamma(\nu_0)([C_{7}+C_0|\Omega|\max\{1,C_2\|\rho^{-1}_{i^*}\|_{\C^{\nu}([0,t^*])}\}]\mathcal{R}
+n^*|\mathcal{U}(0,n^*)|)}$$ and $T_{I}$ given by \eqref{2.8} with
$\varepsilon_{I}$ satisfying \eqref{e.e}.

 Besides, if the order $\nu_1$ is given (i.e.
$\nu_1=\nu_{1,a}$), then, for any $\varepsilon_{II},\lambda\in(0,1)$
and each $\varepsilon\in(0,1-\lambda^{\varepsilon_{II}})$,
\eqref{2.6.2} holds for every $t_a\in(0,T_{II}]$ with
\begin{equation*}\label{2.19}
T_{II}=\min\Big\{t^*,
(2n^*)^{-1/\nu_0},\Big[\frac{C_9\varepsilon}{1+n^*C_9\varepsilon}\Big]^{2/\alpha\nu_1}\Big\}.
\end{equation*}
\end{theorem}
\begin{remark}\label{r2.3}On the one hand, letting
\[ \varepsilon_\nu=\begin{cases}
1-\nu_{i^*+1},\quad
\text{if}\quad i^*\neq M,\\
1-\underline{\nu},\qquad\quad \text{if}\quad i^*=M,
\end{cases}
\]
we can exclude the dependence on $\nu_{1,a}$ in the value of
$\varepsilon_{\nu}$ in Theorem \ref{t2.4} but, on the other hand,
such a choice increases the error  $\varepsilon_{II}$.
\end{remark}
\begin{remark}\label{re.e}
The straightforward computations along with Corollary \ref{c.f1}
show that assumption \eqref{e.e} on $\varepsilon_{I}$ guarantees not
only $\nu_{1,a}<1$ but also $\varepsilon_{\nu}<1$, the latter makes
the  set for $\varepsilon_{II}$ nonempty.
\end{remark}
\begin{remark}\label{r2.5}
Clearly, inequality \eqref{2.15} in the assumption h7 is implicit
requirement on the given data. In Section \ref{s7}, we discuss
sufficient conditions on the data providing \eqref{2.15} and present
an example of both the observation $\psi$ and  the corresponding
initial-boundary value problem which generates this measurement.
\end{remark}

At this point, coming to the SIP and denoting
\begin{align}\label{2.23}\notag
\c_3:&=\c_3(t)=\int_{\Omega}g(x,t)dx+a_0(t)\psi(t)-\mathcal{I}(t),\\
\notag
 \mathcal{R}_2:&=\mathcal{R}_2
(\varphi,b_0,\mathcal{R})=\delta\max\{2,|\partial\Omega|\}\|\varphi\|_{\C^{\alpha,\alpha/2}(\partial\Omega_{t^*})}
+C_0|\Omega|\|b_0\|_{\C^{\alpha,\alpha/2}(\bar{\Omega}_{t^*})}\mathcal{R}(g,u_0,\varphi),\\
 \mathcal{R}_3:&=\mathcal{R}_3
(g,\varphi,a_0,\mathcal{R})=|\Omega|\|g\|_{\C^{\alpha,\alpha/2}(\bar{\Omega}_{t^*})}+\max\{2,
|\partial\Omega|\}\|\varphi\|_{\C^{\alpha,\alpha/2}(\partial\Omega_{t^*})}\\
\notag &
+C_0|\Omega|\|a_0\|_{\C^{\alpha,\alpha/2}(\bar{\Omega}_{t^*})}\mathcal{R}(g,u_0,\varphi),
\end{align}
we claim.
\begin{lemma}\label{l2.2}
Let assumptions of Theorem \ref{t2.2} along with h7 hold. Then, for
$\alpha_4\in(0,\min\{\alpha/2,\nu_1\}]$ and
$\alpha_0\in(0,\tfrac{\alpha\nu_1}{2}]$, there are the following
estimates:

\noindent (i) $\|\c_1\|_{\C^{\alpha_4}([0,t^*])}\leq
\mathcal{R}_2(\varphi,b_0,\mathcal{R})$;

\noindent (ii) $\|\c_3\|_{\C^{\alpha_4}([0,t^*])}\leq
\mathcal{R}_3(g,\varphi,a_0,\mathcal{R})$;

\noindent (iii)
$\|\mathcal{F}_\gamma\|_{\C^{\alpha_0}([0,t^*])}=\|\c_3-\c_\nu\|_{\C^{\alpha_0}([0,t^*])}\leq
\mathcal{R}_3(g,\varphi,a_0,\mathcal{R})+C_3\mathcal{R}(g,u_0,\varphi)$;

\noindent (iv)
$\|\mathcal{F}_{\gamma,a}\|_{\C^{\alpha_4}([0,t_{1}^*])}\leq C_6
\mathcal{R}_1(\rho_1,\mathcal{R},\mathbf{D}_{t}^{\nu_{1,a}}\psi)+
\mathcal{R}_3(g,\varphi,a_0,\mathcal{R})+C_3\mathcal{R}(g,u_0,\varphi)$
with $\alpha_4\in(0,\min\{\alpha_1,\tfrac{\alpha\nu_1}{2}\}]$.
\end{lemma}
The verification of this claim is a pretty standard but albeit very
technical where the key tools are  Lemma \ref{l2.1}, Theorem
\ref{t2.2} along with the regularity of the corresponding functions
stated in h3-h5 and h7. So we are left it for interested readers.

In further analysis regarding \eqref{2.6.3}, we need an additional
requirement for the memory kernel $\mathcal{K}_0$, which is as
follows.
\begin{description}
 \item[h8]   We assume that
 $\mathcal{K}_0\in\C^{\alpha_5}([0,t^*])$ for some $\alpha_5\in(0,1).$
     \end{description}
 \begin{theorem}\label{t2.5}
Let $\varepsilon_{I}\in(0,1-\bar{\nu}),$ $\gamma\in(\bar{\gamma},1)$
for some given $\bar{\gamma}\in(0,1)$, and let requirements of
Theorem \ref{t2.2} along with h7 and h8 hold. Then for any
$\mu\in(0,1)$ and $\varepsilon_{III}\in(1-\bar{\gamma},1)$, and each
$\varepsilon\in(0,1-\mu^{\frac{\varepsilon_{III}+\bar{\gamma}-1}{3}}),$
the bound \eqref{2.6.3} is true for each $t_a\in(0,T_{III}]$ with
some positive $T_{III}$ satisfying the inequality
\[
T_{III}\leq \min\{T_{I}^{0},T_{K},t_1^*\},
\]
where $T_{I}^{0}$ and $T_K$ are given by \eqref{2.10*} and
\eqref{2.10}, respectively.

\noindent Moreover, if $M\geq 2,$ then
\begin{equation*}
T_{III}=\min\Big\{t_1^*,T_{I},T_{K},\Big(\frac{|\c_1(0)||\mathcal{K}_{0}(0)|\Gamma(1+x^*)\varepsilon}{3[\langle\mathcal{K}_0\rangle_{t,[0,t^*]}^{(\alpha_5)}|\c_{1}(0)|
+\|\mathcal{K}_{0}\|_{\C([0,t^*])}\mathcal{R}_2]}
\Big)^{1/\alpha_6},
\Big(\frac{\varepsilon|\mathcal{F}_{\gamma,a}(0)|}{3[C_3\mathcal{R}+C_6\mathcal{R}_1+\mathcal{R}_3]}\Big)^{1/\alpha_7}
 \Big\}
\end{equation*}
with
$\alpha_6=\min\{\alpha_5,\tfrac{\alpha}{2},\tfrac{2\nu_2}{2-\alpha}\}$
and  $\alpha_7=\min\{\alpha_1,\tfrac{\alpha\nu_2}{2}\}$.

If the order $\nu_1$ is known (i.e. $\nu_1=\nu_{1,a}$), then for any
$\varepsilon_{III},\mu\in(0,1)$ and each
$\varepsilon\in(0,1-\mu^{\varepsilon_{III}})$,  estimate
\eqref{2.6.3} holds for every $t_{a}\in(0,T_{III}]$, where
\begin{equation*}
T_{III}=\min\Big\{t^*,\Big(\frac{|\c_1(0)||\mathcal{K}_{0}(0)|\Gamma(1+x^*)\varepsilon}{3[\langle\mathcal{K}_0\rangle_{t,[0,t^*]}^{(\alpha_5)}|\c_{1}(0)|
+\|\mathcal{K}_{0}\|_{\C([0,t^*])}\mathcal{R}_2]} \Big)^{1/\alpha_6}
 \Big\}
\end{equation*}
 with
$\alpha_6=\min\{\alpha_5,\tfrac{\alpha}{2},\tfrac{2\nu_1}{2-\alpha}\}$.
 \end{theorem}
\begin{remark}\label{r2.f}
Finally, we remark, that in the case of the I type FDO, the
regularity of the coefficients $\rho_i(t)$, $i-1,...,M,$ can be
relaxed. Namely, we need in $\rho_i\in \C^{\alpha/2}[0,T]$ (see for
details \cite[Remark 4.5]{PSV1} and \cite[Theorem 3.1]{SV1}).
\end{remark}

The proof of Theorems \ref{t2.3}-\ref{t2.5}, which are rather
technical, will be postponed to the forthcoming Sections
\ref{s4}-\ref{s6}, while the analysis  of condition \eqref{2.15} in
h7 will be carried out in Section \ref{s7}.

\section{Auxiliary Results}
\label{s3}

\noindent Here, we establish some technical results playing a
crucial role in the proof of main results. We start by considering
the function $F=F(t):[0,T^*]\to\R$ vanishing at $t=0$ (i.e.
$F(0)=0$) and possessing  the following properties: \textit{for each
$\varepsilon^*\in(0,1),$ there exists a positive
$t_{\varepsilon}:=t(\varepsilon^*)\in(0,T^*]$ such that}
\begin{equation}\label{3.1}
F\in\C([0,t_{\varepsilon}])\quad and \quad |F(t)|\leq \varepsilon^*
\quad for\,\, any\quad t\in[0,t_{\varepsilon}].
\end{equation}
Performing simple technical calculations, we proceed with the two
straightforward claims.
\begin{corollary}\label{c3.1}
Let $\varepsilon^*\in(0,1)$ and $F$ meet requirements \eqref{3.1},
then
\begin{equation*}\label{3.2}
|\ln|1+F(t)||\leq |\ln(1-\varepsilon^*)|
\end{equation*}
for all $t\in[0,t_{\varepsilon}]$.

\noindent If in additionally $T^{*}\leq 1$, then for each
$\varepsilon_1^*\in(0,1)$ there is the bound
\begin{equation*}\label{3.3}
\frac{|\ln|1+F(t)||}{|\ln t|}\leq\varepsilon^*_{1}
\end{equation*}
for all $t\in[0,t_{1}^*]$ with
\[
t_1^*=\min\{T^*,t_{\varepsilon},(1-\varepsilon^*)^{1/\varepsilon_1^*}\}.
\]
\end{corollary}
\begin{corollary}\label{c3.2}
Let $T^*\leq 1,$ $\varepsilon_2^*\in(0,1)$ and $\lambda\in(0,1)$. If
$F$ satisfies \eqref{3.1} with
$\varepsilon^*=\varepsilon^*(\lambda,\varepsilon_2^*)\in(0,1-\lambda^{\varepsilon_2^*})$,
then  the bound
\[
|\log_{\lambda}|1+F(t)||\leq \varepsilon_2^*
\]
holds for each $t\in[0,t_2^*]$ with
$t_2^*=t_{\varepsilon^*(\lambda,\varepsilon_2^*)}=t(\varepsilon^*(\lambda,\varepsilon_2^*))$.
\end{corollary}
The following two statements will be used in the further analysis to
verify Theorem \ref{t2.3}.
\begin{corollary}\label{c3.3}
Let a positive $T^*\leq 1$ be given, and let a continuous function
$w=w(t):[0,T^*]\to\R$ admit the following representation
\[
w(t)-w(0)=\frac{C_1^* t^{\theta}}{\Gamma(1+\theta)}+w_{1}(t)
\]
for each $t\in[0,T^*]$, where
  $\theta\in(0,1)$,  $C_{1}^*\neq 0$, and
 a continuous function $w_1=w_1(t):[0,T^*]\to\R$ satisfies  the
 inequality
 \[
|t^{-\theta}w_1(t)|\leq C_2^* t^{\theta^*}
 \]
 for all $t\in[0,T^*]$, where $\theta^*$ and $C^{*}_2$ are positive
 constants.

 \noindent Then for any  $\varepsilon^*,\varepsilon_3^*\in(0,1)$, the
 estimates
 \begin{equation*}\label{3.4}
\Big|\frac{\Gamma(1+\theta)w_1(t)}{t^{\theta}C_1^*}\Big|\leq
\varepsilon^*\qquad\text{and}\qquad
\Big|\theta-\frac{\ln|w(t)-w(0)|}{\ln t}\Big|\leq \varepsilon_3^*
 \end{equation*}
 hold for each positive $t$ satisfying the bound
 \begin{equation}\label{3.5}
t<\min\Big\{T^*,(|C_1^*|\Gamma(1+x^*))^{2/\varepsilon_3^*},(\Gamma(1+x^*)/|C_1^*|)^{2/\varepsilon_3^*},
(1-\varepsilon^*)^{2/\varepsilon_3^*},
(|C_1^*|\varepsilon^*/C_2^*)^{1/\theta^*} \Big\}.
 \end{equation}
\end{corollary}
\begin{proof}
Performing straightforward  calculations and utilizing the
representation of $w(t)$ arrive at the inequality
\[
\Big|\theta-\frac{\ln|w(t)-w(0)|}{\ln t}\Big|\leq
\frac{|\ln|C_1^*||+|\ln\Gamma(1+x^*)|}{|\ln t|}+
\Big|\frac{\ln|1+\Gamma(1+\theta)w_1(t)t^{-\theta}/C_1^*|}{\ln
t}\Big|.
\]
At this point,  for each given $\varepsilon^*$ and
$\varepsilon_3^*$, we easily deduce that
\[
\frac{|\ln|C_1^*||+|\ln\Gamma(1+x^*)|}{|\ln t|}\leq
\frac{\varepsilon_3^*}{2},
\]
for every positive $t$ satisfying the inequality
\[
t<\min\Big\{T^*,(|C_1^*|\Gamma(1+x^*))^{2/\varepsilon_3^*},(\Gamma(1+x^*)/|C_1^*|)^{2/\varepsilon_3^*}
\Big\},
\]
while
\[
|\Gamma(1+\theta)w_1(t)t^{-\theta}/C_1^*|\leq \varepsilon^*
\]
for every $t\in[0,(|C_1^*|\varepsilon^*/C_2^*)^{1/\theta^*})$.

\noindent Then,  the last inequality along with  Corollary
\ref{c3.1} with $F=\Gamma(1+\theta)w_1(t)t^{-\theta}/C_1^*$ provide
the estimate
\[
\Big|\frac{\ln|1+\Gamma(1+\theta)w_1(t)t^{-\theta}/C_1^*|}{\ln
t}\Big|\leq \frac{\varepsilon_3^*}{2}
\]
for all $t\in[0,t_1^*]$, where we set
\[
t_1^*=\min\Big\{T^*, (1-\varepsilon^*)^{2/\varepsilon_3^*},
(|C_1^*|\varepsilon^*/C_2^*)^{1/\theta^*} \Big\}.
\]
In conclusion, collecting all inequalities, we  arrive at the
desired bounds  if positive $t$ satisfies \eqref{3.5}.
\end{proof}
Next, for given numbers $\mu_k,$ $k=0,1,...,K,$ $0<\mu_k<\mu_0<1$,
and given functions $v=v(t):[0,T]\to\R$ and
$r_{k}=r_{k}(t):[0,T]\to\R$
 having the corresponding fractional derivatives, we
set
\begin{align*}
J_{\mu}(v,t)&=\int_{0}^{t}(t-\tau)^{\mu-1}[\mathbf{D}_{\tau}^{\mu}v(\tau)-\mathbf{D}_{\tau}^{\mu}v(0)]d\tau,\quad\mu\in(0,1),\\
\mathbf{D}_{t}^{I}v(t)&=\sum_{k=0}^{K}r_{k}(t)\mathbf{D}_{t}^{\mu_k}v(t),\qquad
\mathbf{D}_{t}^{II}v(t)=\sum_{k=0}^{K}\mathbf{D}_{t}^{\mu_k}(r_{k}(t)v(t)),
\end{align*}
\begin{lemma}\label{l3.1}
Let a positive given $T^*<1$,  $0<\mu_k<\mu_0<1,$ $k=1,...,K,$ and
$\mu_k\neq\mu_i,$ if $k\neq i.$ We assume that

\noindent(i) $r_k\in\C^{1}([0,T^*])$ and $r_0>0$ for all
$t\in[0,T^*];$

\noindent(ii) a continuous  function $v$ on $[0,T^*]$ has fractional
derivatives $\mathbf{D}_{t}^{\mu_k}v\in\C^{\mu^*}([0,T^*])$ with
$\mu^*\in(0,\mu_0]$.

Then, for any $\varepsilon^*,\varepsilon_4^*\in(0,1)$ and each
$t\in[0, t_2^*(\varepsilon^*)],$ the following estimates hold
\begin{align*}
\Big|\mu_0-\frac{\ln|v(t)-v(0)|}{\ln
t}\Big|&<\varepsilon_4^*,\quad\text{if}\quad
\mathbf{D}_{t}^{I}v(0)\neq 0,\\
\Big|\mu_0-\frac{\ln|r_0(t)v(t)-r_0(0)v(0)|}{\ln
t}\Big|&<\varepsilon_4^*,\quad\text{if}\quad
\mathbf{D}_{t}^{II}v(0)\neq 0,\
\end{align*}
where
\[
t_2^*(\varepsilon^*)=\begin{cases} \min\Big\{T^*,
\Big(\frac{\Gamma(1+x^*)|\mathbf{D}_t^{I}v(0)|}{r_0(0)}\Big)^{\frac{2}{\varepsilon_4^*}},
\Big(\frac{\Gamma(1+x^*)r_0(0)}{|\mathbf{D}_t^{I}v(0)|}\Big)^{\frac{2}{\varepsilon_4^*}},
(1-\varepsilon^*)^{\frac{2}{\varepsilon_4^*}},\Big(\frac{\varepsilon^*|\mathbf{D}_t^{I}v(0)|}{r_0(0)C_3^*}\Big)^{\frac{1}{\nu^*}}
\Big\}\,\text{if }\, \mathbf{D}_{t}^{I}v(0)\neq 0,\\
\\
 \min\Big\{T^*,
(\Gamma(1+x^*)|\mathbf{D}_t^{II}v(0)|)^{\frac{2}{\varepsilon_4^*}},
\Big(\frac{\Gamma(1+x^*)}{|\mathbf{D}_t^{II}v(0)|}\Big)^{\frac{2}{\varepsilon_4^*}},
(1-\varepsilon^*)^{\frac{2}{\varepsilon_4^*}},\Big(\frac{\varepsilon^*|\mathbf{D}_t^{II}v(0)|}{C_3^*}\Big)^{\frac{1}{\nu^*}}
\Big\}\,\text{if }\,\mathbf{D}_{t}^{II}v(0)\neq 0,
\end{cases}
\]
with $\nu^*=\min\{\mu^*,\mu_0-\mu_1,\mu_0-\mu_2,...,\mu_0-\mu_{K}\}$
and
\[
C_3^*=\begin{cases}
\frac{\|\mathbf{D}_t^{\mu_0}v\|_{\C^{\mu^*}([0,T^*])}}{\Gamma(1+x^*)}\Big(1+\sum\limits_{k=1}^{K}\frac{|r_k(0)|}{r_0(0)\Gamma(1+x^*)}\Big)
+
\sum\limits_{k=1}^{K}\frac{|r_k(0)|\langle\mathbf{D}_t^{\mu_k}v\rangle_{t,[0,T^*]}^{(\mu^*)}}{r_0(0)\Gamma^2(1+x^*)},
\quad\text{if}\quad \mathbf{D}_{t}^{I}v(0)\neq 0,\\
\\
\frac{\langle\mathbf{D}_t^{\mu_0}r_0v\rangle_{t,[0,T^*]}^{(\mu^*)}}{\Gamma(1+x^*)}+
\sum\limits_{k=1}^{K}\frac{\|\mathbf{D}_t^{\mu_0}r_kv\|_{\C([0,T^*])}}{\Gamma^2(1+x^*)}
+
\sum\limits_{k=1}^{K}\frac{\langle\mathbf{D}_t^{\mu_k}r_kv\rangle_{t,[0,T^*]}^{(\mu^*)}}{\Gamma^2(1+x^*)},\qquad\quad\text{if
}\quad \mathbf{D}_{t}^{II}v(0)\neq 0.
\end{cases}
\]
\end{lemma}
\begin{proof}
In this proof, assuming $\mathbf{D}_{t}^{I}v(0)\neq 0,$  we obtain
the first estimate stated in this claim, being the second  verified
with the analogous arguments. On this route, the key tool  is
Corollary \ref{c3.3}. Indeed, under assumptions (i)-(ii),
\cite[Lemma 4.1]{PSV2} provides the following asymptotic
\begin{align*}
[v(t)-v(0)]r_0(0)\Gamma(1+\mu_0)&=t^{\mu_0}\mathbf{D}_{t}^{I}v(0)+\mu_0r_0(0)J_{\mu_0}(v,t)+\sum_{k=1}^{K}\mu_kr_k(0)t^{\mu_0-\mu_k}J_{\mu_k}(v,t)\\
&-[v(t)-v(0)]\sum_{k=1}^{K}r_{k}(0)t^{\mu_0-\mu_k}\Gamma(1+\mu_k),
\end{align*}
if $\mathbf{D}_{t}^{I}v(0)\neq 0.$ Hence, accounting this
representation and properties of the functions $v$ and $r_k$, we
utilize Corollary \ref{c3.3} with
\begin{align*}
C_1^*&=\frac{\mathbf{D}_{t}^{I}v(0)}{r_0(0)},\quad
\theta=\mu_0,\qquad w(t)=v(t),\\
w_1(t)&=\frac{\mu_0J_{\mu_0}(v,t)}{\Gamma(1+\mu_0)}+\sum_{k=1}^{K}\frac{\mu_k
r_k(0)
t^{\mu_0-\mu_k}}{r_0(0)\Gamma(1+\mu_0)}J_{\mu_k}(v,t)-\frac{[v(t)-v(0)]}{r_0(0)\Gamma(1+\mu_0)}\sum_{k=1}^{K}r_k(0)t^{\mu_0-\mu_k}\Gamma(1+\mu_k).
\end{align*}
The latter  completes immediately  the proof of this lemma, only if
 the introduced  $w_1$ meets the
requirements of Corollary \ref{c3.3}. Appealing to the regularity of
$v$ and performing the direct computations, we have
\begin{align*}
|J_{\mu_0}(v,t)|&\leq
t^{\mu_0+\mu^*}\langle\mathbf{D}_{t}^{\mu_0}v\rangle_{t,[0,T^*]}^{(\mu^*)}\int_{0}^{1}(1-z)^{\mu_0-1}z^{\mu^*}dz=
t^{\mu_0+\mu^*}\langle\mathbf{D}_{t}^{\mu_0}v\rangle_{t,[0,T^*]}^{(\mu^*)}\frac{\Gamma(\mu_0)\Gamma(1+\mu^*)}{\Gamma(1+\mu_0+\mu^*)},\\
|J_{\mu_k}(v,t)|&\leq
t^{\mu_k+\mu^*}\langle\mathbf{D}_{t}^{\mu_k}v\rangle_{t,[0,T^*]}^{(\mu^*)}\frac{\Gamma(\mu_k)\Gamma(1+\mu^*)}{\Gamma(1+\mu_k+\mu^*)},\\
|v(t)-v(0)|&\leq\frac{t^{\mu_0}}{\Gamma(1+\mu_0)}\|
\mathbf{D}_{t}^{\mu_0}v\|_{\C([0,T^*])}.
\end{align*}
At last, exploiting these inequalities to handle the terms in the
representation of $w_1$, we end up with the chain of the bounds
\begin{align*}
t^{-\mu_0}|w_1(t)|&\leq
\sum_{k=1}^{K}\frac{|r_k(0)|}{r_{0}(0)\Gamma^2(1+x^*)}[t^{\mu^*}\langle\mathbf{D}_{t}^{\mu_k}v\rangle_{t,[0,T^*]}^{(\mu^*)}
+t^{\mu_0-\mu_k}\|\mathbf{D}_{t}^{\mu_0}v\|_{\C([0,T^*])}]\\
&+\frac{t^{\mu^*}}{\Gamma(1+x^*)}\langle\mathbf{D}_{t}^{\mu_0}v\rangle_{t,[0,T^*]}^{(\mu^*)}
 \leq t^{\nu^*}C_3^{*},
\end{align*}
which in turn provides the desired results.
\end{proof}
\begin{remark}\label{r3.1}
The proof of Lemma \ref{l3.1} tells us that the positive constant
$C_3^*$ in  $t_2^*(\varepsilon)$ can be modified if there is a
positive constant $C_4^{*}$ such that
\[
\|\mathbf{D}_{t}^{\mu_0}v\|_{\C^{\mu^*}([0,T^*])}+\sum_{k=1}^{K}\|\mathbf{D}_{t}^{\mu_k}v\|_{\C^{\mu^*}([0,T^*])}\leq
C_{4}^*.
\]
Namely, in this case, appealing to \cite[Lemma 5.3, Remark
5.1]{SV1}, we obtain the estimate
\[
\sum_{k=0}^{K}\|\mathbf{D}_t^{\mu_k}(r_kv)\|_{C^{\mu^*}([0,T^*])}\leq
C_4^*C_5^*\sum_{k=0}^{K}\|r_{k}\|_{\C^{1}([0,T^*])}
\]
with the positive value $C_{5}^{*}$ being independent of $\mu_{k}$.

As a result, accounting this inequality in Lemma \ref{l3.1} allows
us to set
\[
C_3^*=C_{4}^{*}\cdot\begin{cases}
1+2\sum\limits_{k=1}^{K}\frac{|r_k(0)|}{r_0(0)\Gamma^2(1+x^*)},
\qquad\qquad\qquad\text{if}\quad \mathbf{D}_{t}^{I}v(0)\neq 0,\\
\\
C_5^*\Big[1+\frac{2}{\Gamma(1+x^*)}\Big]\sum\limits_{k=0}^{K}\|r_{k}\|_{\C^{1}([0,T^*])}
,\qquad\text{if }\quad \mathbf{D}_{t}^{II}v(0)\neq 0.
\end{cases}
\]
\end{remark}
The next two  claims deal with  the convolution of a smooth function
$f=f(t):\C[0,T^{*}]\to\R$, with  weak singular kernels. The first
convolution is introduced as
\begin{equation}\label{g.1}
\mathcal{G}_f(t):=(\mathcal{G}f)(t,\gamma_3,n)=\int\limits_{0}^{1}(1-z)^{\gamma_3-1}nE_{\gamma_3,\gamma_3}(-nt^{\gamma_3}(1-z)^{\gamma_3})f(zt)dz,
\end{equation}
where  $\gamma_3\in(0,1),$ $n\in\mathbb{N}$,  and
$E_{\gamma_3,\gamma_3}(z)$ is the two-parametric Mittag-Leffler
function defined as
\begin{equation}\label{3.6}
E_{\theta_1,\theta_2}(z)=\sum_{k=0}^{+\infty}\frac{z^k}{\Gamma(\theta_1k+\theta_2)},\quad
\theta_1,\theta_2>0,\quad z\in\R.
\end{equation}
It is worth noting that this convolution can be also considered as
the integral operator defined on the space of continuous function.
We notice that the main properties of this operator including
compactness, description of the spectrum, nondecreasing and some
integral inequalities associated with this operator are described in
\cite{DJ}.
\begin{lemma}\label{l3.2}
Let given positive $T^*<1$ and $\gamma_3\in(0,1),$
 $n\in\mathbb{N}$, and let
$f\in\C^{\gamma_4}([0,T^*])$ with some $\gamma_4\in(0,\gamma_3)$,
$f(0)\neq 0$. Then for any $\lambda,\varepsilon_{5}^{*}\in(0,1)$ and
each $\varepsilon^*\in(0,1-\lambda^{\varepsilon_{5}^{*}})$, the
estimate
\[
\Big|\log_{\lambda}\Big|\frac{\mathcal{G}_{f}(\lambda
t)}{\mathcal{G}_{f}(t)}\Big|\Big|\leq \varepsilon_{5}^{*}
\]
holds for each $t\in[0,t_3^*]$, where
\begin{align*}
t_3^*&:=
t_3^*(\varepsilon^*)=\min\Big\{T^*,(2n)^{-1/\gamma_3},\Big[\frac{C_6^*(\varepsilon^*)}{1+nC_6^*(\varepsilon^*)}\Big]^{1/\gamma_4}\Big\}\quad\text{with}\\
C_{6}^*(\varepsilon^*)&=\frac{\Gamma(1+x^*)|f(0)|\varepsilon^*}{3\Gamma(\gamma_4)[\langle
f\rangle_{t,[0,T^*]}^{(\gamma_4)}+n|f(0)|]}.\end{align*}
\end{lemma}
\begin{proof}
First, we rewrite the
 desired estimate  in more suitable form
\begin{equation*}\label{3.7}
\Big|\log_{\lambda}\Big|1+\frac{\mathcal{G}_{f}(\lambda
t)-\mathcal{G}_{f}(t)}{\mathcal{G}_{f}(t)}\Big|\Big|<\varepsilon_5^*,
\end{equation*}
which tells us that  the proof of Lemma \ref{l3.2} is a
straightforward employing Corollary \ref{c3.2} with
\[
F(t):=\frac{\mathcal{G}_{f}(\lambda
t)-\mathcal{G}_{f}(t)}{\mathcal{G}_{f}(t)},\qquad t_2^*:=t_3^*\qquad
\text{and}\quad \varepsilon_2^*:=\varepsilon_5^*.
\]
Indeed, to achieve this, we are left to examine the following
statements:

\noindent(i) $ F(0)=\frac{\mathcal{G}_{f}(\lambda
t)-\mathcal{G}_{f}(t)}{\mathcal{G}_{f}(t)}|_{t=0}=0,$

\noindent(ii) for  any $\lambda,\varepsilon_5^*\in(0,1)$ and each
$\varepsilon^*\in(0,1-\lambda^{\varepsilon_5^*}),$ $F$ is continuous
on $[0,t_3^*],$

\noindent(iii) for  any $\lambda,\varepsilon_5^*\in(0,1)$ and each
$\varepsilon^*\in(0,1-\lambda^{\varepsilon_5^*})$, there is the
bound
\begin{equation}\label{3.8}
|F(t)|<\varepsilon^*\quad \text{for any}\quad t\in[0,t_3^*].
\end{equation}
We note that the point (i) in these statements provides the desired
estimate in this lemma if $t=0$.

At this point, we  verify each statement in  (i)-(iii), separately.

\noindent$\bullet$ As for (i), exploiting  the easily verified
relations
\[
E_{\gamma_3,\gamma_3}(0)=\frac{1}{\Gamma(\gamma_3)}\quad
\text{and}\quad nt^{\gamma_3}(1-z)<1\quad \text{for}\quad
z\in(0,1),\quad t\in[0,n^{-1/\gamma_3})\, \text{ and any fixed }\,
n\in\mathbb{N},
\]
and appealing to nonvanishing $f(0)$, we arrive at the relations
\begin{equation*}\label{3.8**}
\mathcal{G}_{f}(\lambda
t)|_{t=0}=\mathcal{G}_{f}(t)|_{t=0}=\frac{nf(0)}{\Gamma(1+\gamma_3)}\neq
0,
\end{equation*}
which in turn end up  with the equality stated in (i).

\noindent$\bullet$ Coming to (ii), the regularity of $f$ along with
the properties of the Mittag-Leffler function tell that
$\mathcal{G}_{f}(t)$ and $\mathcal{G}_{f}(\lambda
t)-\mathcal{G}_{f}(t)$ are continuous if
$t\in[0,\min\{T^*,(2n)^{-1/\gamma_3}\}]$. The latter means the
continuity of $F(t)$ on $[0,t_3^*]$ if only
\[
\mathcal{G}_{f}(t)\neq 0\quad\text{for all}\quad t\in[0,t_3^*].
\]
To examine this inequality, employing the easily verified equality
\[
E_{\gamma_3,\gamma_3}(z)-E_{\gamma_3,\gamma_3}(0)=z
E_{\gamma_3,2\gamma_3}(z)\quad \forall \quad z\in\R,
\]
we rewrite $\mathcal{G}_{f}(t)$ in more appropriate form for the
further analysis
\begin{align*}
\mathcal{G}_{f}(t)&=\int\limits_{0}^{1}n
z^{\gamma_3-1}E_{\gamma_3,\gamma_3}(-nt^{\gamma_3}z^{\gamma_3})[f(t(1-z))-f(0)]dz\\
&-f(0)n^2 t^{\gamma_3}\int\limits_0^1
z^{2\gamma_3-1}E_{\gamma_3,2\gamma_3}(-nt^{\gamma_3}z^{\gamma_3})dz+\frac{nf(0)}{\Gamma(1+\gamma_3)}.
\end{align*}
Then, appealing to the completely monotonicity of
$E_{\theta_1,\theta_2}(-z)$ for $z\geq 0,$ $\theta_1\in[0,1],$
$\theta_2\geq \theta_1$, and  employing the straightforward verified
inequality
\begin{equation}\label{3.8*}
E_{\theta_1,\theta_2}(-z)\leq
\Gamma^{-1}(1+x^{*})(1-z)^{-1}\quad\text{for any}\quad z\in[0,1),
\end{equation}
we arrive at the chain of inequalities
\begin{align*}\label{3.9}\notag
|\mathcal{G}_{f}(t)|&\geq \frac{n|f(0)|}{\Gamma(1+\gamma_3)}\Big|1-n
t^{\gamma_3}\Gamma(1+\gamma_3)\int\limits_0^{1}z^{2\gamma_3-1}E_{\gamma_3,2\gamma_3}(-nt^{\gamma_3}z^{\gamma_3})dz
\\\notag &
-\frac{t^{\gamma_4}\Gamma(1+\gamma_3)\langle
f\rangle_{t,[0,T^*]}^{(\gamma_4)}}{|f(0)|}\int\limits_0^{1}z^{\gamma_3-1}(1-z)^{\gamma_4}E_{\gamma_3,\gamma_3}(-nt^{\gamma_3}z^{\gamma_3})dz
\Big|\\ \notag &=
\frac{n|f(0)|}{\Gamma(1+\gamma_3)}\Big|1-\frac{t^{\gamma_4}\Gamma(1+\gamma_3)\Gamma(\gamma_3)\langle
f\rangle_{t,[0,T^*]}^{(\gamma_4)}}{|f(0)|}E_{\gamma_3,\gamma_4+\gamma_3}(-nt^{\gamma_3})\\
\notag & -n
t^{\gamma_3}\Gamma(1+\gamma_3)E_{\gamma_3,1+2\gamma_3}(-nt^{\gamma_3})
\Big|\\
& \geq \frac{n|f(0)|}{\Gamma(1+\gamma_3)}
\Big|1-\frac{t^{\gamma_4}\Gamma(\gamma_4)}{(1-nt^{\gamma_3})\Gamma(1+x^*)}
\Big[ \frac{\langle
f\rangle_{t,[0,T^*]}^{(\gamma_4)}}{|f(0)|}+n\Big]\Big|,
\end{align*}
whenever
\begin{equation}\label{3.9*}
\begin{cases}
t\leq \min\{T^{*},(2n)^{-1/\gamma_3}\},\\
\\
\frac{t^{\gamma_4}}{1-nt^{\gamma_3}}\leq
\frac{3C_6^{*}(\varepsilon^*)}{\varepsilon^*}.
\end{cases}
\end{equation}
Here, to compute the integrals, we used formulas (4.4.4)-(4.4.5) in
\cite{GKMR}.

As for \eqref{3.9*}, bearing in mind the restriction on $\gamma_4$
and $T^*$, we perform the direct calculations and arrive at
\eqref{3.9*} if $t\in[0,t_3^*]$ for any
$\lambda,\varepsilon_5^*\in(0,1)$ and each
$\varepsilon^*\in(0,1-\lambda^{\varepsilon_5^*})$. In conclusion,
taking into account the restriction on $\varepsilon^*,$ we have
\begin{equation}\label{3.10}
|\mathcal{G}_{f}(t)|\geq
\frac{n|f(0)|}{\Gamma(1+\gamma_3)}[1-\varepsilon^*/3]>\frac{n|f(0)|[2+\lambda^{\varepsilon_5^*}]}{3\Gamma(1+\gamma_3)},
\end{equation}
if $t\in[0,t_3^*]$. Hence, \eqref{3.10} completes the verification
of (ii).

\noindent$\bullet$ Coming to the bound \eqref{3.8} and accounting
the arguments of the previous step, we rewrite the difference
$[\mathcal{G}_{f}(\lambda t)-\mathcal{G}_{f}(t)]$ in the form
\[
\mathcal{G}_{f}(\lambda t)-\mathcal{G}_{f}(t)=g_1(\lambda
t)-g_1(t)+g_2(t)
\]
with
\begin{align*}
g_1(t)&=n\int\limits_{0}^{1}
z^{\gamma_3-1}E_{\gamma_3,\gamma_3}(-nt^{\gamma_3}z^{\gamma_3})[f(t(1-z))-f(0)]dz,\\
g_2(t)&=f(0) n\int\limits_{0}^{1}
z^{\gamma_3-1}[E_{\gamma_3,\gamma_3}(-n\lambda^{\gamma_3}t^{\gamma_3}z^{\gamma_3})-
E_{\gamma_3,\gamma_3}(-nt^{\gamma_3}z^{\gamma_3})]dz.
\end{align*}
Next,  the regularity of $f(t)$ along with \eqref{3.8*} and formula
(4.4.5) in \cite{GKMR}   arrive at the estimates
\begin{align}\label{3.12}\notag
|g_1(t)|+|g_1(\lambda t)|&\leq n t^{\gamma_4}\Gamma(\gamma_3)\langle
f\rangle_{t,[0,T^*]}^{(\gamma_4)}[E_{\gamma_3,1+\gamma_3+\gamma_4}(-nt^{\gamma_3})
+\lambda^{\gamma_4}E_{\gamma_3,1+\gamma_3+\gamma_4}(-n\lambda^{\gamma_3}t^{\gamma_3})]\\\notag
& \leq\frac{n t^{\gamma_4}\Gamma(\gamma_4)\langle
f\rangle_{t,[0,T^*]}^{(\gamma_4)}}{\Gamma(1+x^*)}\Big[\frac{\lambda^{\gamma_4}}{1-n(\lambda
t)^{\gamma_3}}+\frac{1}{1-n t^{\gamma_3}}\Big]\\
\notag & \leq \frac{2n t^{\gamma_4}\Gamma(\gamma_4)\langle
f\rangle_{t,[0,T^*]}^{(\gamma_4)}}{(1-nt^{\gamma_3})\Gamma(1+x^*)},
\end{align}
if $t\leq \min\{T^*,(2n)^{-1/\gamma_3}\}$.

As for $g_2(t),$ we use (4.4.4) in \cite{GKMR} and the properties of
the Mittag-Leffler functions (see \eqref{3.6} and \eqref{3.8*}) and
end up with the inequalities
\begin{align*}
|g_2(t)|&\leq
n|f(0)||E_{\gamma_3,1+\gamma_3}(-n\lambda^{\gamma_3}t^{\gamma_3})-E_{\gamma_3,1+\gamma_3}(-nt^{\gamma_3})|\\
&
=n|f(0)|\Big|\sum\limits_{k=1}^{\infty}\frac{(-nt^{\gamma_3})^{k}(\lambda^{\gamma_3
k}-1)}{\Gamma(1+\gamma_3+k\gamma_3)}\Big|,
\end{align*}
which in turn entail
\begin{equation*}\label{3.13}
|g_2(t)|\leq \frac{2n^2
t^{\gamma_3}|f(0)|}{\Gamma(1+x^*)(1-nt^{\gamma_3})}.
\end{equation*}
Finally, collecting all estimates of $g_i$ yields
\begin{equation}\label{3.14}
|\mathcal{G}_{f}(\lambda t)-\mathcal{G}_{f}(t)|\leq
\frac{2n\varepsilon^*
t^{\gamma_4}|f(0)|}{3C_6^*(1-nt^{\gamma_3})}\leq
\frac{2n\varepsilon^*|f(0)|}{3}
\end{equation}
whenever $t\leq t_3^{*}$.

At last, letting $t\in[0,t^*_3]$ and exploiting \eqref{3.10} and
\eqref{3.14}, we end up with the desired bound
\[
|F(t)|=\frac{|\mathcal{G}_{f}(\lambda
t)-\mathcal{G}_{f}(t)|}{|\mathcal{G}_{f}(t)|}\leq
\frac{2\varepsilon^*}{3-\varepsilon^*}<\varepsilon^*.
\]
It is worth noting that the last inequality holds only if
$\varepsilon^*\in(0,1)$, the latter is performed automatically due
to the assumptions on the parameters  $\lambda$ and
$\varepsilon^*_5$.

Summing up, we conclude that all the assumptions  on $F$ required in
Corollary \ref{c3.2} are satisfied, which completes the proof of
Lemma \ref{l3.2}.
\end{proof}
\begin{remark}\label{r3.0}
The proof of Lemma \ref{l3.2} tells us that  $f$ may depend not only
$t$ but  also on $n\in\mathbb{N}$ as a parameter. Namely, if
$f(0,n)\neq 0$ and $f\in C^{\gamma_4}([0,T^*])$ for each fixed $n\in
\mathbb{N}$, then the results of Lemma \ref{l3.2} hold with
$f=f(t,n)$.
\end{remark}
Next, we aim to get the similar results to the convolution in more
general form
\begin{equation}\label{g.2}
G_{f}(t):=G(t;k,f)=\int\limits_{0}^{1}(1-z)^{\gamma^*-1}k(t-zt)f(zt)dz
\end{equation}
with given continuous functions $f$, $k$.
\begin{lemma}\label{l3.3}
Let $T^*\in(0,1),$ $\gamma^*\in(0,1)$ and let
$k\in\C^{\gamma_{3}}([0,T^*]),$ $f\in\C^{\gamma_4}([0,T^*]),$
$\gamma_3,\gamma_4\in(0,1)$. If $k(0)\neq 0$ and $f(0)\neq 0$, then
for any $\lambda,\varepsilon_6^*\in(0,1)$ and each
$\varepsilon^*\in(0,1-\lambda^{\varepsilon_6^{*}}),$ the following
estimate
\[
\Big|\log_{\lambda}\Big|\frac{G_f(\lambda
t)}{G_f(t)}\Big|\Big|<\varepsilon_6^*
\]
holds for any $t\in[0,t_4^*]$ with
$t_4^*:=t_4^*(\varepsilon^*)=\min\{T^*,(C_7^{*})^{1/\bar{\gamma}}\}$,
where
\[
\bar{\gamma}=\min\{\gamma_3,\gamma_4\}\qquad\text{and}\qquad
C_{7}^{*}=\frac{\varepsilon^*|f(0)||k(0)|\Gamma(1+x^*)}{3[|f(0)|\langle
k\rangle_{t,[0,T^*]}^{(\gamma_3)}+\langle
f\rangle_{t,[0,T^*]}^{(\gamma_4)}\|k\|_{\C([0,T^*])}]}.
\]
\end{lemma}
\begin{proof}
Arguing similarly to the proof of the previous claim, we conclude
that Lemma \ref{l3.3} follows immediately from  Corollary \ref{c3.2}
if we let
\[
F(t):=\frac{G_f(\lambda t)-G_f(t)}{G_f(t)}, \quad
t_2^*:=t_4^*\quad\text{and}\quad \varepsilon_2^*:=\varepsilon_6^*,
\]
where $F$ meets the requirements:

\noindent(i) $\frac{G_f(\lambda t)-G_f(t)}{G_f(t)}|_{t=0}=0,$

\noindent(ii) for any fixed $\lambda,\varepsilon_6^*\in(0,1)$ and
each $\varepsilon^*\in(0,1-\lambda^{\varepsilon_6^*}),$ the
introduced above function $F$ is continuous on $[0,t_4^*]$ and,
besides,
\begin{equation}\label{3.15}
\frac{|G_f(\lambda t)-G_f(t)|}{|G_f(t)|}<\varepsilon^*\qquad
\text{for all}\quad t\in[0,t_4^*].
\end{equation}
It is apparent that, the regularity of $k$ and $f$ along with the
inequalities $f(0)\neq 0,$ $k(0)\neq 0$ provide the continuity of
$G_f(t)$ and $G_{f}(\lambda t)$ for each $t\in[0,T^*]$ and
$\lambda\in(0,1)$, and, besides,
\[
G_f(0)=\frac{k(0)f(0)}{\gamma^*}\neq 0.
\]

Thus, we are left to verify inequalities (3.15) and $G_{f}(t)\neq 0$
for any $t\in[0,t_4^*]$. Assuming  $t\in[0,t_4^*]$ and performing
the straightforward computations, we arrive at the chain of the
inequalities
\begin{align*}
|G_f(t)|&\geq
\frac{|k(0)f(0)|}{\gamma^*}\Big[1-\frac{\gamma^*}{|k(0)|}\int\limits_{0}^1
z^{\gamma^*-1}|k(zt)-k(0)|dz-\frac{\gamma^*}{|k(0)||f(0)|}\int\limits_{0}^1
z^{\gamma^*-1}|k(tz)||f(t-tz)-f(0)|dz \Big]\\
& \geq \frac{|k(0)f(0)|}{\gamma^*} \Big| 1- \frac{\gamma^*
t^{\gamma_3}\langle
k\rangle^{(\gamma_3)}_{t,[0,T^*]}}{|k(0)|(\gamma^*+\gamma_3)}
 -
\frac{\Gamma(1+\gamma^*)\Gamma(1+\gamma_4) t^{\gamma_4}\langle
f\rangle^{(\gamma_4)}_{t,[0,T^*]}\|k\|_{\C([0,T^*])}}{|k(0)||f(0)|\Gamma(1+\gamma^*+\gamma_4)}
 \Big|\\&
 \geq
\frac{|k(0)f(0)|}{\gamma^*}\Big|1-\frac{t^{\bar{\gamma}}}{|f(0)||k(0)|\Gamma(1+x^*)}[\langle
k\rangle^{(\gamma_3)}_{t,[0,T^*]}|f(0)|+\langle
f\rangle^{(\gamma_4)}_{t,[0,T^*]}\|k\|_{\C([0,T^*])}]\Big|.
\end{align*}
Thanks to $ t^{\bar{\gamma}}\leq C_7^* $ and
$\varepsilon^*<1-\lambda^{\varepsilon_6^{*}}$, we can continue the
evaluation of $G_f(t)$  and write
\begin{equation}\label{3.16}
|G_f(t)|\geq \frac{|k(0)f(0)|}{3\gamma^*}|3-\varepsilon^*|>
\frac{|k(0)f(0)|}{3\gamma^*}(2+\lambda^{\varepsilon_6^*}),
\end{equation}
which in turn means that $G_f(t)\neq 0$ for all $t\in[0,t_4^*]$.

In fine, we are left to verify \eqref{3.15}. Setting
\[
\bar{g}_1(t)=\int\limits_0^1
z^{\gamma^*-1}k(tz)[f(t-tz)-f(0)]dz,\quad
\bar{g}_2(t)=f(0)\int\limits_0^1 z^{\gamma^*-1}[k(\lambda
tz)-k(tz)]dz,
\]
we rewrite the numerator in the left-hand side of \eqref{3.15} in
the form
\[
|G_f(\lambda t)-G_f(t)|=|\bar{g}_1(\lambda
t)-\bar{g}_1(t)+\bar{g}_2(t)|.
\]
Then, appealing to the regularity of $f$ and $k$ and performing the
technical calculations, we end up with the estimates for each
$t\in[0,T^*]$:
\begin{align*}
|\bar{g}_1(t)|&\leq \langle
f\rangle_{t,[0,T^*]}^{(\gamma_4)}\|k\|_{\C([0,T^*])}t^{\gamma_4}\int\limits_{0}^{1}z^{\gamma^*-1}(1-z)^{\gamma_4}dz\\
&\leq \frac{\Gamma(\gamma^*)\langle
f\rangle_{t,[0,T^*]}^{(\gamma_4)}\|k\|_{\C([0,T^*])}}{\Gamma(1+x^*)}t^{\bar{\gamma}},\\
|\bar{g}_2(t)|&\leq\frac{|f(0)| \langle
k\rangle_{t,[0,T^*]}^{(\gamma_3)}t^{\bar{\gamma}}}{\gamma^*},
\end{align*}
which in turn entail
\begin{align*}\label{3.18}\notag
|G_f(\lambda t)-G_f(t)|&\leq \frac{2
t^{\bar{\gamma}}}{\gamma^*\Gamma(1+x^*)}(\langle
f\rangle_{t,[0,T^*]}^{(\gamma_4)}\|k\|_{\C([0,T^*])}+|f(0)|\langle
k\rangle_{t,[0,T^*]}^{(\gamma_3)})\\
& =\frac{2\varepsilon^* t^{\bar{\gamma}}|f(0)k(0)|}{3\gamma^*
C_7^*}\leq \frac{2|f(0)k(0)|\varepsilon^*}{3\gamma^* }
\end{align*}
for all $t\leq t_4^*$.

At last, collecting  the last estimate with \eqref{3.16} and letting
$t\leq t_4^*,$ we deduce that
\[
\frac{|G_f(\lambda t)-G_f(t)|}{|G_f(t)|}\leq
\frac{2\varepsilon^*}{3-\varepsilon^*}<\varepsilon^* ,
\]
 where the last inequality holds only if $\varepsilon^*<1$ that is
 provided by restrictions on $\lambda$ and $\varepsilon^*_{6}$.
 Thus, this finishes the proof of Lemma \ref{l3.3}.
\end{proof}


\section{Proof of Theorem \ref{t2.3}}
\label{s4}

\noindent First, we prove this theorem in simpler case dealing with
only  the reconstruction of $\nu_1$ (i.e. assuming that the
remaining parameters and data in \eqref{i.1}-\eqref{i.3} are given).
After that, we demonstrate how the results getting in the particular
case can be extended to the general one.

\subsection{Theorem \ref{t2.3} in  a simpler case: the reconstruction  of
$\nu_1$ only}\label{s4.1}

In this case, we focus on the following inverse problem:

\textit{For the given  right-hand sides in \eqref{i.2}, \eqref{i.3},
the memory kernel $\mathcal{K}$; the coefficients in the operators
$\mathbf{D}_t,$ $\mathcal{L}_i,$ $\mathcal{N}$ and orders $\nu_i,$
$i=2,...,M,$ the inverse problem consists in the reconstruction of
the couple $(\nu_1,u)$ such that}
\begin{equation}\label{4.1}
\nu_1\in\Big(\frac{2\nu_2}{2-\alpha},1\Big),\quad \text{if}\quad
M\geq 2,\qquad\text{while}\quad \nu_1\in(0,1),\quad\text{if}\quad
M=1,
\end{equation}
\textit{and $u$ solves (in a classical sense) of problem
\eqref{i.1}-\eqref{i.3} and satisfies \eqref{i.4} for all
$t\in[0,t^*]$.}

Obviously, exploiting \cite[Lemma 3.1]{HPV} and recasting the
arguments of \cite[Section 4]{HPV} with $\mathcal{K}$ being given,
we end up with the solvability of this inverse problem.
\begin{lemma}\label{l4.1}
Let $T>0$ be finite and  assumptions h1-h6 and \eqref{4.1} hold, and
let $\c_{\nu,0}\neq 0$. Then the inverse problem
\eqref{i.1}-\eqref{i.4} has a unique solution $(\nu_1,u)$ such that
$\nu_1$ is computed by \eqref{i.13} and
$u\in\C^{2+\alpha,\frac{2+\alpha}{2}\nu_1}(\bar{\Omega}_{T})$ is a
unique global classical solution of the direct problem
\eqref{i.1}-\eqref{i.3}, which satisfies the observation \eqref{i.4}
for $t\in[0,t^*]$. Moreover, \eqref{2.1}-\eqref{2.2*} hold.
\end{lemma}

At this point, based on Lemmas \ref{l4.1} and \ref{l3.1} and Remark
\ref{r3.1}, we evaluate the time interval for which estimate
\eqref{2.6.1} is true in the case of the aforementioned  inverse
problem. Namely, we can use Lemma \ref{l3.1} and then Remark
\ref{r3.1} with
\[
v=\psi,\quad r_i=\rho_{i+1}, \quad T^*=t^*,\quad
\mu_i=\nu_{i+1},\quad i=0,...,M-1,
\]
and
\[
C_4^*=\frac{C_0\mathcal{R}(g,u_0,\varphi)}{\Gamma(1+x^*)},\quad
C_5^*=C_1+C_2 ,\quad\varepsilon_4^*=\varepsilon^*=\varepsilon_{I},
\]
where $C_i$, $i=0,1,2,$ are defined in \eqref{2.1} and Corollary
\ref{c2.1}, while $\mathcal{R}(g,u_0,\varphi)$ is given by
\eqref{2.0}. Thus, performing the straightforward technical
calculations we claim.
\begin{lemma}\label{l4.2}
Let assumptions of Lemma \ref{l4.1} hold and $\nu_{1,a}$ be computed
by \eqref{2.3} for each $t_a\in(0,t^*]$. Then, for each
$\varepsilon_{I}\in(0,1),$ the bound \eqref{2.6.1} is true  for
every $t_a\in(0,T_{\nu}]$, where
\[
T_{\nu}=\begin{cases}
\min\Big\{T_{I}^{0},\Big(\frac{\varepsilon_{I}|\c_{\nu,0}|}{\rho_1(0)C_4\mathcal{R}}\Big)^{1/\mu^*}\Big\}\quad\text{for
the I type FDO},\\
\min\Big\{T_{I}^{0},\Big(\frac{\varepsilon_{I}|\c_{\nu,0}|}{C_4\mathcal{R}}\Big)^{1/\mu^*}\Big\}\quad\text{for
the II type FDO}
\end{cases}
\quad\text{ with}\quad \mu^*=\begin{cases}
\frac{\nu_2\alpha}{2},\qquad\text{if}\quad M\geq 2,\\
\frac{\nu_1\alpha}{2},\qquad\text{if}\quad M=1,
\end{cases}
\]
and $T_{I}^{0}$ and $C_4$ are defined via \eqref{2.10*} and
\eqref{2.10**}, respectively.
\end{lemma}
\begin{remark}\label{r4.1}
It is worth noting that the disadvantage of choosing $T_{\nu}$ is
the complexity of explicitly calculating the values $C_{i},i=0,1,2,$
which determine  $C_ 3$. Moreover, in the case of the one-term FDO
(i.e., $M=1$), $\mu^*$ depends on the searched parameter $\nu_1$.
However, in any case, the definition of $T_{\nu}$ suggests that
\[
T_{\nu}\leq T_{I}^{0}.
\]
\end{remark}

\subsection{Completion of the proof of Theorem
\ref{t2.3}}\label{s4.2}
 Obviously, the verification of Theorem \ref{t2.3} in the general
 case, that is for the first and second inverse problems is very
 analogously to arguments leading to Lemma \ref{l4.2}. Namely,
 appealing either to Theorem \ref{t2.1} in the case of the FIP or
to Theorem \ref{t2.2} in the case of the SIP, and then using Lemma
\ref{l3.1} together with Remark \ref{r3.1}, where we set
\[
 \mu^*=\begin{cases}
\frac{\nu_2\alpha}{2},\qquad\text{in the case of the SIP ,}\\
\nu_0,\qquad\qquad\text{in the case of the FIP,}
\end{cases}
\quad \text{and}\quad C_3^*=C_4\mathcal{R},\quad T^*=t^*,\quad
t_2(\varepsilon^*)=T_{I},
\]
we immediately arrive at the relations
\begin{equation*}\label{4.2}
T_I=T_{\nu}\leq T_I^0
\end{equation*}
with $T_\nu$ being defined in Lemma \ref{l4.2}.

Clearly, in the case of the FIP with $M\geq 3,$ assumption h2
together with the restriction on $t^*$ suggest that $\mu^*\geq
\nu_0$ and, accordingly,
\[
T_{\nu}\geq \begin{cases}
\min\Big\{T^0_{I},\Big(\frac{|\c_{\nu,0}|\varepsilon_{I}}{C_4\mathcal{R}\rho_1(0)}\Big)^{1/\nu_0}\Big\}\quad
\text{for the I type FDO,}\\
\min\Big\{T^0_{I},\Big(\frac{|\c_{\nu,0}|\varepsilon_{I}}{C_4\mathcal{R}}\Big)^{1/\nu_0}\Big\},\quad
\text{for the II type FDO.}
\end{cases}
\]
Thus, as $T_{I}$ we can choose the right-hand side of this
inequality.

Coming to the SIP, we first set $\mu^*=\frac{\alpha\nu_2}{2}$ in the
definition of $T_{\nu}$. Next, keeping in mind Remark \ref{r2.2}, we
may select $T_{I}$ in the form \eqref{2.9}, if we obtain two sides
uniformly estimates in $\gamma$ of $\|\mathcal{K}\|_{L^1(0,t_1)}$
with $t_1=\min\{t^*,T_{K}\}$. Indeed, assuming that
$\gamma\in(0,\gamma_0)$ with some fixed $\gamma_0<1$ and appealing
to the smoothness and nonvanishing of the memory kernel
$\mathcal{K}_0(t)$ if $t\in[0,t_1]$, we have the estimates
\begin{equation}\label{4.3}
0<t_1^2 \underset{t\in[0,t_1]}{\min}|\mathcal{K}_0(t)|\leq
\|\mathcal{K}\|_{L^1(0,t_1)}\leq
\frac{1}{1-\gamma_0}\|\mathcal{K}_0\|_{\C([0,t_1^*])},
\end{equation}
which provide the desired results. That completes the proof of
Theorem \ref{t2.3}. \qed


\section{Proof of Theorem \ref{t2.4}}
\label{s5}

\noindent First of all, the arguments of \cite[Sections
3.1-3.2]{HPSV} tell us that the following equality (rewritten here
below in our notation) holds
\[
\nu_{i^*}=\nu_1+\log_{\lambda}\frac{|\mathcal{G}_{\U}(\lambda
t,\nu_1-\nu_{i^*},n^*)|}{|\G_{\U}(t,\nu_1-\nu_{i^*},n^*)|}-\log_{\lambda}\frac{|\F_{\nu}(\lambda
t)|}{|\F_{\nu}(t)|}
\]
for each $t\in[0,t^*]$, $\lambda\in(0,1)$ and $n^*$ ensuring
\eqref{2.16}. Here, the functions $\G_{\U}$ and $\U$ are defined by
\eqref{g.1} and \eqref{2.14}, respectively.

Then, collecting  \eqref{2.4} with this equality, we arrive at the
relation
\begin{equation}\label{5.1}
\nu_{i^*}-\nu_{i^*,a}=\log_{\lambda}\frac{|\mathcal{G}_{\U}(\lambda
t_a,\nu_1-\nu_{i^*},n^*)|}{|\G_{\U}(t_a,\nu_1-\nu_{i^*},n^*)|}+\log_{\lambda}\frac{\lambda^{\nu_1-\nu_{1,a}}|\F_{\nu}(t_a)|}{|\F_{\nu}(\lambda
t_a)|}+ \log_{\lambda}\frac{|\F_{\nu,a}(\lambda
t_a)|}{|\F_{\nu,a}(t_a)|}
\end{equation}
for any $t_a\in(0,t^*].$

In light equality \eqref{5.1} (see also Remark \ref{re.e}), we
easily deduce that Theorem \ref{t2.4} is a simple consequence of the
following bounds for any $\varepsilon_{2,i}\in(0,1),$ $i=1,2,3,$
such that
$\varepsilon_{2,2}\in[\varepsilon_\nu+\varepsilon_{2,1},1)$:
\begin{align}\label{5.2.1}
\Big|\log_{\lambda}\frac{|\mathcal{G}_{\U}(\lambda
t_a,\nu_1-\nu_{i^*},n^*)|}{|\G_{\U}(t_a,\nu_1-\nu_{i^*},n^*)|}\Big|&\leq
\varepsilon_{2,1}\quad\forall\, t_a\in(0,T_{II}^1],\\ \label{5.2.2}
\Big|\log_{\lambda}\frac{\lambda^{\nu_1-\nu_{1,a}}|\F_{\nu}(t_a)|}{|\F_{\nu}(\lambda
t_a)|}\Big|&\leq \varepsilon_{2,2}\quad\forall\, t_a\in(0,T_{II}^2],\\
\label{5.2.3}  \Big|\log_{\lambda}\frac{|\F_{\nu,a}(\lambda
t_a)|}{|\F_{\nu,a}(t_a)|}\Big|&\leq \varepsilon_{2,3}\quad\forall\,
t_a\in(0,T_{II}^3],
\end{align}
for any $\lambda\in(0,1)$ and $T_{II}^{j}\leq t^*$ being specified
below.

\subsection{Estimates \eqref{5.2.1}-\eqref{5.2.2}}\label{s5.1}
We notice that  bound \eqref{5.2.1} follows from Lemmas \ref{l2.1}
and \ref{l3.2}. Indeed, appealing to Lemma \ref{l2.1} and taking
into account nonvanishing $\U(0,n^*)$, we can easily verified that
for
\[
f=\U,\quad\gamma_4=\alpha_0,\quad\gamma_3=\nu_1-\nu_{i^*},\quad
T^{*}=t^*,\quad \varepsilon_5^*=\varepsilon_{2,1},\quad
\varepsilon*=\varepsilon,
\]
 all requirements of
Lemma \ref{l3.2} are satisfied and, hence, we end up with
\eqref{5.2.1} where
\begin{equation}\label{5.4}
T_{II}^1=\min\Big\{t^*,(2n^*)^{1/\alpha_0},\Big(\frac{C_9\varepsilon}{1+n^*
C_9\varepsilon}\Big)^{1/\alpha_0}\Big\}
\end{equation}
with $\varepsilon\in(0,1-\lambda^{\varepsilon_{2,1}})$ and
$\varepsilon_{2,1}\in(0,1).$ At last, we notice that due to
$\nu_0<\frac{\alpha\nu_1}{2}$, we can put $\alpha_0=\nu_0$ in
\eqref{5.4} in the case of $M\geq 3.$

 At this point, coming to the estimate
\eqref{5.2.2} and recasting step-by-step  the arguments of
\cite[Sections 3.1-3.3]{HPSV}, we derive the equality
\begin{equation}\label{5.5}
\F_{\nu}(t)=
t^{\nu_1-\nu_{i^*}}\G_{\U}(t,\nu_1-\nu_{i^*},n^*)\quad\forall\,
t\in[0,t^*].
\end{equation}
Thus, making use this relation along with \eqref{5.2.1}, we arrive
at the desired bound
\begin{align*}
\Big|\log_{\lambda}\frac{\lambda^{\nu_1-\nu_{1,a}}|\F_{\nu}(t_a)|}{|\F_{\nu}(\lambda
t_a)|}\Big|&=\Big|\nu_{i^*}-\nu_{1,a}+
 \log_{\lambda}\frac{|\mathcal{G}_{\U}(\lambda
t_a,\nu_1-\nu_{i^*},n^*)|}{|\G_{\U}(t_a,\nu_1-\nu_{i^*},n^*)|}\Big|\\
&\leq
\varepsilon_{\nu}+\Big|\log_{\lambda}\frac{|\mathcal{G}_{\U}(\lambda
t_a,\nu_1-\nu_{i^*},n^*)|}{|\G_{\U}(t_a,\nu_1-\nu_{i^*},n^*)|}\Big|\\
& \leq \varepsilon_{2,1}+\varepsilon_{\nu} \leq  \varepsilon_{2,2}
\end{align*}
for each $t_a\in(0,T_{II}^2]$ with $ T_{II}^2=T_{II}^1, $ where
$T_{II}^1$ is given by \eqref{5.4} with
$\varepsilon\in(0,1-\lambda^{\varepsilon_{2,1}})$ and
$0<\varepsilon_{2,1}+\varepsilon_{\nu}\leq\varepsilon_{2,2}<1$.

\subsection{Estimate \eqref{5.2.3}}\label{s5.3}
Assuming h7, we will argue similar to the previous step. To this
end, taking into account the representation of $\c_2(t)$ and
$\F_{\nu}(t),$ we rewrite $\F_{\nu,a}$ in more suitable form
\[
\F_{\nu,a}(t)=\F_{\nu}(t)+\c_{2}(t),
\]
which in turn leads to
\begin{equation}\label{5.7}
\F_{\nu,a}(0)=\c_{2}(0)\neq 0.
\end{equation}
Here, we used \eqref{2.15} and \eqref{5.5}.
\begin{proposition}\label{l5.2}
Let assumptions of Theorem \ref{t2.1} along with h7 hold. Then, for
any $\lambda\in(0,1),$ $\varepsilon_{2,3}\in(0,1),$ the estimate
\eqref{5.2.3} holds for each $t_a\in(0,T_{II}^3]$, where
\[
T_{II}^{3}=\min\Big\{t_1^*,\Big(\frac{\varepsilon|\c_{2}(0)|}{3C_8[\mathcal{R}+\mathcal{R}_1]}\Big)^{1/\alpha_2}\Big\},
\]
with each $\varepsilon\in(0,1-\lambda^{\varepsilon_{2,3}})$ and any
$\alpha_2\in(0,\min\{\alpha_1,\alpha\nu_1/2\}]$.
\end{proposition}
\begin{proof}
Obviously, this claim follows immediately from Corollary \ref{c3.2}
being applied to
\[
F(t)=\frac{\F_{\nu,a}(t_a\lambda)-\F_{\nu,a}(t_a)}{\F_{\nu,a}(t_a)},\quad
T^*=t^*,\quad \varepsilon_2^*=\varepsilon_{2,3},\quad
\varepsilon^*=\varepsilon,
\]
only if this function $F$ meets all requirements of Corollary
\ref{c3.2}.

Clearly, assumption h7 together with the requirements of Theorem
\ref{t2.1} allow us to utilized Lemma \ref{l2.1} and, accounting
\eqref{5.7}, to deduce that
\[
F(0)=0\quad \text{and}\quad
\F_{\nu,a}(t_a\lambda),\,\F_{\nu,a}(t_a)\in\C^{\alpha_2}([0,t_1^*]).
\]
Thus, to complete the proof of this claim, we need to show that
\begin{equation*}\label{5.8}
|\F_{\nu,a}(t_a)|>0 \quad\text{and}\quad \frac{|\F_{\nu,a}(\lambda
t_a)-\F_{\nu,a}(t_a)|}{|\F_{\nu,a}(t_a)|}<\varepsilon
\end{equation*}
for each $t_a\in(0,T_{II}^3]$ and any $\lambda\in(0,1),$
$\varepsilon\in(0,1-\lambda^{\varepsilon_{2,3}})$.

\noindent Exploiting Lemma \ref{l2.1} arrives at
\[
|\F_{\nu,a}(\lambda t_a)-\F_{\nu,a}(t_a)|\leq
C_8(\mathcal{R}_1+\mathcal{R}](1-\lambda)^{\alpha_2}t_a^{\alpha_2}\quad\text{for
each}\quad t_a\in[0,t_{1}^{*}].
\]

Coming to evaluation of $|\F_{\nu,a}(t_a)|$ and accounting
\eqref{5.7}, we have
\begin{align*}
|\F_{\nu,a}(t_a)|&\geq
|\F_{\nu,a}(0)|\Big|1-\frac{|\F_{\nu,a}(\lambda
t_a)-\F_{\nu,a}(0)|}{|\F_{\nu,a}(0)|}\Big|\\& \geq
|\c_2(0)|\Big|1-\frac{C_8[\mathcal{R}+\mathcal{R}_1]t_a^{\alpha_2}}{|\c_2(0)|}\Big|\\&
\geq |\c_2(0)|(1-\varepsilon/3),
\end{align*}
if
\[
t_a\leq\min\Big\{t_1^*,\Big(\frac{\varepsilon|\c_2(0)|}{3C_8[\mathcal{R}+\mathcal{R}_1]}\Big)^{1/\alpha_2}\Big\}.
\]
At last, taking into account the last inequality along with Lemma
\ref{l2.1}, we end up with the desired inequality for any positive
$t_a\leq T_{II}^{3}$. That finishes the verification of Proposition
\ref{l5.2}.\end{proof}
\begin{remark}\label{r5.0}
Since $\alpha_2\leq \min\{\alpha_1,\alpha\nu_1/2\}$, we may set
$\alpha_2=\min\{\alpha_1,\nu_0\}$ in the case of $M\geq 3$.
\end{remark}

\subsection*{Conclusion of the proof of Theorem
\ref{t2.4}}\label{s5.4} It is apparent that, if $\nu_1=\nu_{1,a},$
then equality \eqref{5.1} together with the bound \eqref{5.2.1}
leads to  Theorem \ref{t2.4} with
$\varepsilon_{II}=\varepsilon_{2,1}\in(0,1)$ and $T_{II}=T_{II}^{1}$
(see \eqref{5.4}) with $\alpha_0=\alpha\nu_1/2$.

In general case (i.e. $\nu_1\neq \nu_{1,a}$), assuming \eqref{e.e}
and letting
$\varepsilon_{2,1}=\varepsilon_{2,3}=\varepsilon_{2,2}-\varepsilon_{\nu}=\frac{\varepsilon_{II}}{3}$
in \eqref{5.2.1}-\eqref{5.2.3}, we may select
$T_{II}=\min\{T_{II}^1,T_{II}^{2},T_{II}^{3},T_{I}\}$, which,
obviously, finishes the verification of Theorem \ref{t2.4}. \qed


\section{Proof of Theorem \ref{t2.5}}
\label{s6}

\noindent It is worth noting that \cite[Lemma 3.2]{HPV} along with
arguments of \cite[Section 4.1]{HPV} provide not only formula
\eqref{i.15} (to compute $\gamma$) but also give the equality
(rewritten here below in our notations)
\[
\gamma=1+\log_\mu\frac{|\F_{\gamma}(t)|}{|\F_{\gamma}(\mu
t)|}+\log_{\mu}\frac{|G_{\c_1}(\mu t)|}{|G_{\c_1}(t)|},
\]
which is true for each $\mu\in(0,1)$ and $t\in[0,t^*]$ if only the
conditions of Theorem \ref{t2.2} hold. Here, we recall that $\c_1$
and $G_{\c_1}$ are defined via \eqref{c.1} and \eqref{g.2},
respectively, in the latter we put
\[
k=\mathcal{K}_0 ,\quad \gamma^*=1-\gamma,\quad f=\c_1.
\]
Collecting this equality (at $t=t_a$) with \eqref{2.5}, we end up
with
\[
\gamma-\gamma_a=\log_\mu\frac{|\F_{\gamma}(t_a)|}{|\F_{\gamma}(\mu
t_a)|}+\log_{\mu}\frac{|G_{\c_1}(\mu t_a)|}{|G_{\c_1}(t_a)|} +
\log_\mu\frac{|\F_{\gamma,a}(\mu t_a)|}{|\F_{\gamma,a}(t_a)|} .
\]
This equality  suggests that \eqref{2.6.3} in the case of $\nu_1\neq
\nu_{1,a}$ is provided by the following estimate
\begin{equation*}\label{6.1}
\Big|\log_\mu\frac{|\F_{\gamma}(t_a)|}{|\F_{\gamma}(\mu
t_a)|}\Big|+\Big|\log_{\mu}\frac{|G_{\c_1}(\mu
t_a)|}{|G_{\c_1}(t_a)|}\Big| +\Big| \log_\mu\frac{|\F_{\gamma,a}(\mu
t_a)|}{|\F_{\gamma,a}( t_a)|}\Big| <\varepsilon_{III}
\end{equation*}
for all $t_a\in(0,T_{III}]$ and any $\varepsilon_{III}\in(0,1)$.
Besides, $\nu_1=\nu_{1,a}$, then the bound   means
\begin{equation*}\label{6.2}
\Big|\log_{\mu}\frac{|G_{\c_1}(\mu
t_a)|}{|G_{\c_1}(t_a)|}\Big|<\varepsilon_{III}
\end{equation*}
for all $t_a\in(0,T_{III}]$ and any $\varepsilon_{III}\in(0,1)$.

In order to specify $T_{III}$ in the estimates above, we, first,
find $T_{III}^i,$ $i=1,2,,3,$ such that for each
$\varepsilon_{3,i}\in(0,1),$
$1-\bar{\gamma}+\varepsilon_{3,1}\leq\varepsilon_{3,2}$, the
following inequalities hold
\begin{align}\label{6.3.1}
&\Big|\log_{\mu}\frac{|G_{\c_1}(\mu
t_a)|}{|G_{\c_1}(t_a)|}\Big|<\varepsilon_{3,1}\qquad\forall\,
t_a\in(0,T_{III}^1],\\
\label{6.3.2} &
\Big|\log_\mu\frac{|\F_{\gamma}(t_a)|}{|\F_{\gamma}(\mu
t_a)|}\Big|<\varepsilon_{3,2}\qquad\, \forall\, t_a\in(0,T_{III}^2],\\
\label{6.3.3} &\Big| \log_\mu\frac{|\F_{\gamma,a}(\mu
t_a)|}{|\F_{\gamma,a}( t_a)|}\Big| <\varepsilon_{3,3}\qquad\forall\,
t_a\in(0,T_{III}^3].
\end{align}
At this point, to evaluate each  $T_{III}^{i}$, we analyze each
bound \eqref{6.3.1}-\eqref{6.3.2}, separately.

\subsection{Estimates \eqref{6.3.1}-\eqref{6.3.2}}\label{s6.1}
First, we discuss \eqref{6.3.1} in a simpler case, when $\nu_1$ is
given explicitly, that is
\begin{equation}\label{6.4}
\nu_1=\nu_{1,a}.
\end{equation}
Then, we describe how this restriction can be removed.
\begin{proposition}\label{p6.1}
Let assumption of Theorem \ref{t2.2} along with h8 and \eqref{6.4}
hold. Then, for any $\varepsilon_{3,1},\mu\in(0,1),$ the estimate
\eqref{6.3.1} is true for every $t_a\in(0,T_{\gamma}]$ with
\[
T_{\gamma}=\min\Big\{t^*,\Big(\frac{\varepsilon
|\c_{1,0}||\mathcal{K}_0|\Gamma(1+x^*)}{3[\mathcal{R}_2\|\mathcal{K}_0\|_{\C([0,t^*])}+
|\c_{1,0}|\langle\mathcal{K}_0\rangle_{t,[0,t^*]}^{(\alpha_5)}]}\Big)^{1/\mathrm{a}}\Big\},
\]
where $\mathrm{a}\in(0,\min\{\alpha_5,\alpha/2,\nu_1\}],$ and
$\varepsilon\in(0,1-\mu^{\varepsilon_{3,1}})$.
\end{proposition}
\begin{proof}
The verification of this claim is simple usage of Lemma \ref{l3.3}
with
\[
k=\mathcal{K}_0,\quad f=\c_1,\quad\lambda=\mu,\quad
T=t^*,\quad\gamma^*=1-\gamma,\quad
\varepsilon_6^*=\varepsilon_{3,1},\quad \varepsilon^*=\varepsilon.
\]
To this end, we just need to specify $\gamma_3$ and $\gamma_4$ such
that $\mathcal{K}_0\in\C^{\gamma_3}([0,t^*])$ and
$\c_{1}\in\C^{\gamma_4}([0,t^*])$. Obviously, assumption h8 along
with Lemma \ref{l2.2} ensure
$\mathcal{K}_0\in\C^{\alpha_5}([0,t^*])$ and
$\c_1\in\C^{\alpha_4}([0,t^*])$. Hence, all requirements of Lemma
\ref{l3.3} are satisfied and, accordingly, this lemma arrives at the
desired bound \eqref{6.3.1} with $T_{III}^1=T_{\gamma}$.
\end{proof}

It is worth noting that the quantity $T_\gamma$ is depending on
$\nu_1$ (see the definition of $\mathrm{a}$). This fact tells us
that if $\nu_1$ is unknown parameter, we have to get rid of this
dependence. Clearly, it is possible if $\D_t$ is at least a two-term
fractional differential operator. Indeed, in this case, appealing to
h2, we conclude that
\[
\min\{\alpha_5,\alpha/2,\nu_1\}>\min\{\alpha_5,\alpha/2,2\nu_2/(2-\alpha)\}.
\]
Next, thanks to $t^*<1$, we can use Proposition \ref{p6.1} without
restriction \eqref{6.4} and end up with the following claim.
\begin{lemma}\label{l6.1}
Let assumptions of Theorem \ref{t2.2} and h8 hold. Then for any
$\varepsilon_{3,1},\mu\in(0,1),$ the bound \eqref{6.3.1} is true for
each $t_a\in(0,T_{III}^{1}]$, where $T_{III}^{1}\leq t^*$ and,
besides, in the case of $M\geq 2,$ there is
\[
T_{III}^1=T_{\gamma}
\]
with $\mathrm{a}=\min\{\alpha_5,\alpha/2,2\nu_2/(2-\alpha)\}.$
\end{lemma}

Coming to estimate \eqref{6.3.2}, we notice that  the key tool in
the verification of this bound is the following equality established
in \cite[Section 4]{HPV} (see (4.2),(4.3) and Lemma 3.2 therein)
\[
\F_{\gamma}(\mu t)=(\mu
t)^{1-\gamma}\int_{0}^{1}z^{-\gamma}\mathcal{K}_0(\mu tz)\c_1(\mu
t(1-z))dz
\]
for all $\mu\in[0,1],$ $t\in[0,t^*]$.

 Indeed, exploiting this
relation, we arrive at the relations
\begin{align*}
\Big|\log_{\mu}\frac{|\F_{\gamma}(t)|}{|\F_{\gamma}(\mu
t)|}\Big|&=\Big|\gamma-1+\log_{\mu}\frac{\int_{0}^{1}z^{-\gamma}\mathcal{K}_0(tz)\c_1(t(1-z))dz}
{\int_{0}^{1}z^{-\gamma}\mathcal{K}_0(\mu tz)\c_1(\mu
t(1-z))dz}\Big|\\
&=\Big|\gamma-1+\log_{\mu}\frac{|G_{\c_1}(t)|}{|G_{\c_1}(\mu
t)|}\Big|\leq |1-\bar{\gamma}|+\varepsilon_{3,1}.
\end{align*}
Thanks to the restriction on $\varepsilon_{3,1}$ and
$\varepsilon_{3,2}$, the last inequality leads to the
 \eqref{6.3.2} with $T_{III}^{2}=T_{III}^{1}$.

\subsection{Estimate \eqref{6.3.3}}\label{s6.3}
Here, analogously to arguments of the previous subsection, we
exploit Corollary \ref{c3.2} and Lemma \ref{l2.2}. On this route, we
need additional properties of the function $\F_{\gamma,a}(t),$
described in the following statement and easily verified using Lemma
\ref{l2.2}.
\begin{proposition}\label{p6.3}
Let assumptions of Theorem \ref{t2.2} and h7 hold. Then the function
$\F_{\gamma,a}(t)$ computed via \eqref{2.5*} satisfies the following
relations

\noindent (i) $\F_{\gamma,a}(t)=\F_{\gamma}(t)+\c_2(t)$ for all
$t\in[0,t^*_1]$ with $\c_2$ being defined in \eqref{2.13};

\noindent (ii) $\F_{\gamma,a}(0)=\c_2(0)\neq 0.$
\end{proposition}

\begin{lemma}\label{l6.3}
Let $\varepsilon_{I}\in(0,1-\bar{\nu})$. Under assumptions of
Proposition \ref{p6.3}, the bound \eqref{6.3.3} holds for all
$\mu,\varepsilon_{3,3}\in(0,1)$ and any $t_a\in(0,T_{III}^{3}]$,
where
\[
T_{III}^{3}=\min\Big\{t_1^*,\Big(\frac{\varepsilon|\F_{\gamma,a}(0)|}{3[C_3\mathcal{R}+C_6\mathcal{R}_1+\mathcal{R}_3]}\Big)^{1/\alpha_7},T_I\Big\}
\]
with $\alpha_7\in(0,\min\{\alpha_1,\alpha\nu_1/2\}]$ and
$\varepsilon\in(0,1-\mu^{\varepsilon_{3,3}})$, and $T_I$ being
defined in Theorem \ref{t2.3}.
\end{lemma}
\begin{proof}
It is apparent that Proposition \ref{p6.3} along with Lemma
\ref{l2.2} provide the continuity of $\F_{\gamma,a}(t)$ on
$[0,t_1^*]$ and the equality
\[
\frac{\F_{\gamma,a}(\mu
t_a)-\F_{\gamma,a}(t_a)}{\F_{\gamma,a}(t_a)}\Big|_{t_a=0}=0.
\]
Thus, in order to apply Corollary \ref{c3.2},
 which in turn ends up with the desired result, we are left to obtain the estimate
 \begin{equation}\label{6.12}
\Big|\frac{\F_{\gamma,a}(\mu
t_a)-\F_{\gamma,a}(t_a)}{\F_{\gamma,a}(t_a)}\Big|<\varepsilon
 \end{equation}
 for any $t_a\leq T_{III}^{3}$.

 \noindent Lemma \ref{l2.2} and Proposition \ref{p6.3} yield
 \[
|\F_{\gamma,a}(t_a)|\geq
|\F_{\gamma,a}(0)|\Big|1-\frac{t_a^{\alpha_4}[C_3\mathcal{R}+C_6\mathcal{R}_1+\mathcal{R}_3]}{|\F_{\gamma,a}(0)|}
\geq |\F_{\gamma,a}(0)|(1-\varepsilon/3),
 \]
 if
 \begin{equation}\label{6.13}
t_a\in\Big(0,\min\Big\{t_1^*,\Big(\frac{\varepsilon
|\F_{\gamma,a}(0)|}{3[C_3\mathcal{R}+C_6\mathcal{R}_1+\mathcal{R}_3]}\Big)^{1/\alpha_4}\Big\}\Big].
 \end{equation}
 In fine, collecting the last estimate with  Lemma
 \ref{l2.2}, we arrive at
 \[
\Big|\frac{\F_{\gamma,a}(\mu
t_a)-\F_{\gamma,a}(t_a)}{\F_{\gamma,a}(t_a)}\Big|\leq
\frac{3t_a^{\alpha_4}(1-\mu)^{\alpha_4}[C_3\mathcal{R}+C_6\mathcal{R}_1+\mathcal{R}_3]}{|\F_{\gamma,a}(0)|(3-\varepsilon)}<\frac{\varepsilon}{3-\varepsilon}
 \]
 for each $t_a$ satisfying \eqref{6.13} and $\varepsilon\in(0,1).$
 Thus, the last estimate ends up with \eqref{6.12}, which completes
 the proof of this lemma.
\end{proof}

\subsection{Completion of the proof of Theorem
\ref{t2.5}}\label{s6.4} Now, the proof of Theorem \ref{t2.5} follows
immediately from Proposition \ref{p6.1}
 and Lemmas \ref{l6.1}-\ref{l6.3}.
 Indeed, in the case of given $\nu_1$, i.e. $\nu_1=\nu_{1,a}$,
 Proposition \ref{p6.1} with $\varepsilon_{3,1}=\varepsilon_{III}$
 arrives at \eqref{2.6.3} for all
 $\mu,\varepsilon_{III},\gamma\in(0,1)$ and $t_a\in(0,T_{\gamma}]$.
 The latter deduces Theorem \ref{t2.5} with
 $T_{III}=T_{\gamma}$  in the special case \eqref{6.4}.

 Coming to the general case ($\nu_1\neq\nu_{1,a}$) and taking into
 account the restriction on $\varepsilon_{I}$ (see Corollary \ref{c.f1}), we conclude that $T_{III}\leq
 T_I^0$ with $T_I^0$ given by \eqref{2.10*} with
 $\varepsilon_I\in(0,1-\bar{\nu})$. Appealing to  this fact and
 employing Lemmas \ref{l6.1}-\ref{l6.3} with
$\varepsilon_{3,1}=\varepsilon_{3,3}=\varepsilon_{3,2}+\bar{\gamma}-1=\frac{\varepsilon_{III}-1+\bar{\gamma}}{3}$,
we end up with the desired result with
$T_{III}=\min\{T_{III}^1,T_{III}^2,T_{III}^3,T_{I}\}$. In
conclusion, we remark that in the case of $M\geq 2$, we may select
$\alpha_{6}=\min\{\alpha_5,\alpha/2,\frac{2\nu_2}{2-\alpha}\}$,
$\alpha_{7}=\min\{\alpha_1,\alpha\nu_2/2\}$ in $T_{III}^{3}$ (see
Lemma \ref{l6.3}).  That finishes the verification of Theorem
\ref{t2.5}. \qed

\section{Justification of inequality \eqref{2.15}}
\label{s7}

\noindent It is worth noting that assumption h7 is an implicit
condition on the measurement $\psi$ and on the right-hand sides in
\eqref{i.2}-\eqref{i.3}, coefficients in the corresponding
operators, the memory kernel and the orders of the fractional
derivatives. In this section, we first describe sufficient
conditions on the given data which provide the fulfillment of
\eqref{2.15} in h7. After that, we give an explicit example of
\eqref{i.1}-\eqref{i.3}, which generates $\psi(t)$ such that all
requirements in Theorems \ref{t2.3}-\ref{t2.5} (and in particular
h7) are satisfied.
\begin{lemma}\label{l7.1}
Let $\D_{t}^{\nu_{1,a}}\psi\in\C^{\alpha_1}([0,t_1^*])$ and
$\nu_{1,a}<\nu_1$ with $\nu_{1,a}$ being computed via \eqref{2.3}.
Then, under assumptions of either Theorem \ref{t2.1} or Theorem
\ref{t2.2}, the following inequalities hold
\begin{equation}\label{7.1}
\D_{t}^{\nu_1}\psi(t)|_{t=0}\neq 0, \quad
\D_{t}^{\nu_1}(\rho_1(t)\psi(t))|_{t=0}\neq 0\quad\text{but}\quad
\D_{t}^{\nu_{1,a}}\psi(t)|_{t=0}=
\D_{t}^{\nu_{1,a}}(\rho_1(t)\psi(t))|_{t=0}=0,
\end{equation}
and, besides, \eqref{2.15} holds.
\end{lemma}
\begin{proof}
The verification of this claim follows from   Lemmas 5.2 and 5.4 in
\cite{SV1}, which tell that under assumptions of Lemma \ref{l7.1}
(i.e. $\nu_{1,a}<\nu_1$)
\[
\D_{t}^{\nu_{1,a}}\psi(t)|_{t=0}=\D_{t}^{\nu_{1,a}}(\rho_1(t)\psi(t))|_{t=0}=0\quad\text{and}\quad
\D_{t}^{\nu_{1}}\psi(t)|_{t=0}=\rho^{-1}_{1}(0)\D_t\psi|_{t=0}.
\]
In force of \eqref{2.2*}, we have
\[
\D_{t}^{\nu_{1}}\psi(t)|_{t=0}=\rho^{-1}_{1}(0)\D_t\psi|_{t=0}=\rho^{-1}_{1}(0)\c_{\nu,0}\neq
0.
\]
It is worth noting that the last inequality is provided by the
nonvanishing of $\c_{\nu,0}$, which stated in assumptions of
Theorems \ref{t2.1} and \ref{t2.2}.

In order to verify nonvanishing of
$\D_{t}^{\nu_{1}}(\rho_1(t)\psi(t))|_{t=0}$, we apply
\cite[Proposition 5.5]{SV} and obtain
\[
\D_{t}^{\nu_{1}}(\rho_1(t)\psi(t))=\rho_1(t)\D_{t}^{\nu_1}\psi+\psi(0)\D_{t}^{\nu_1}\rho_1(t)
+\frac{\nu_1}{\Gamma(1-\nu_1)}\int\limits_{0}^{t}\frac{[\rho_1(t)-\rho_1(s)][\psi(s)-\psi(0)]}{(t-s)^{1+\nu_1}}ds.
\]
At last, appealing the smoothness of $\rho_1$ and $\psi$ and
exploiting \cite[Lemmas 5.2-5.3]{SV1}, we arrive at the relations:
\[
\D_{t}^{\nu_1}\rho_1(0)=0,\quad
\int\limits_{0}^{t}\frac{[\rho_1(t)-\rho_1(s)][\psi(s)-\psi(0)]}{(t-s)^{1+\nu_1}}ds\bigg|_{t=0}=0,
\]
which in turn end up with the equalities
\[
\D_{t}^{\nu_{1}}(\rho_1(t)\psi(t))|_{t=0}=\rho_1(0)\D_{t}^{\nu_{1}}\psi(t)|_{t=0}=\c_{\nu,0}
\]
leading to  the desired results.
\end{proof}
\begin{remark}\label{r7.0}
Corollary \ref{c.f1} suggests that if $t_1^*<T_{I}$ with
$\varepsilon_{I}\in(0,1-\bar{\nu})$, then $\nu_{1,a}<1$, and the
assumption $\nu_{1,a}<\nu_1$ makes sense.
\end{remark}
\begin{remark}\label{r7.1}
It is apparent that the function
\begin{equation}\label{7.2}
\psi(t)=c_{\nu}t^{\nu_1}+\sum_{i=0}^{m}c_{i}t^{i}\quad\text{with}\quad
m\in\mathbb{N}\quad \text{and}\quad c_{i}\in\R,\, c_{\nu}\neq 0
\end{equation}
is a simple example of the observation satisfying the conditions
\[
\D_t^{\theta}\psi|_{t=0}=0\quad\forall\, \theta\in(0,\nu_1),\quad
\D_{t}^{\nu_1}\psi|_{t=0}=c_{\nu}, \quad\text{and}\quad
\D_t^{\theta}\psi\in\C^{\nu_1}([0,T])
\]
for any finite $T>0$. Thus, selecting $\theta=\nu_{1,a}<\nu_1$, we
conclude that the measurement \eqref{7.2}
 satisfies h7 only if $\nu_{1,a}<\nu_1$.
\end{remark}

As for the initial-boundary value problem which generates $\psi(t)$
having form \eqref{7.2} with  $\nu_{1,a}<\nu_1$, we refer to
\cite[Example 7.2]{HPV}. For reader's convenience, we report it (in
our notation) here below.

\noindent Setting $\Omega=(0,2)\times(0,2)$ and $T=1$, we easily
examine that for any $\nu_1\in(0,1)$, the function
\begin{equation}\label{7.3}
\psi(t)=\frac{256}{225}(2+t^{\nu_1})=\int\limits_{(0,2)\times(0,2)}v(x_1,x_2,t)dx_1dx_2,
\end{equation}
 where $v$ solves the following problem
\begin{equation}\label{7.4}
\begin{cases}
\frac{1}{2}\D_{t}^{\nu_1}v-\frac{1+t^2}{4}\D_{t}^{\nu_1/5}v-\Delta
v-\mathcal{K}*(\Delta
v+4v)=\sum_{i=1}^{3}g_{i}\quad\text{in}\quad \Omega_{T},\\
v|_{t=0}=2x_1^2x_2^2(2-x_1)^2(2-x_2)^2\quad\text{in}\quad\bar{\Omega},\\
\frac{\partial v}{\partial N}=0\quad\text{on}\quad
\partial\Omega_T,
\end{cases}
\end{equation}
where $\mathcal{K}=t^{-\gamma}(1+t),$ $\gamma\in(0,1)$, and
\begin{align*}
g_1&=x_1^2x_2^2(2-x_1)^2(2-x_2)^2\bar{g}_{1}(t),\\
g_2&=-4[x_2^2(2-x_2)^2(2-6x_1+3x_1^2)+x_1^2(2-x_1)^2(2-6x_2+3x_2^2)]\bar{g}_2(t),\\
g_3&=-4[x_2^2(2-x_2)^2(2-6x_1+3x_1^2)+x_1^2(2-x_1)^2(2-6x_2+3x_2^2)+x_1^2x_2^2(2-x_1)^2(2-x_2)^2]\bar{g}_3(t)
\end{align*}
with
\begin{align*}
\bar{g}_1(t)&=\frac{\Gamma(1+\nu_1)}{2}-\frac{(1+t^2)t^{4\nu_1/5}\Gamma(1+\nu_1)}{4\Gamma(1+4\nu_1/5)},\qquad
\bar{g}_2(t)=2+t^{\nu_1},\\
 \bar{g}_{3}(t)&=\frac{2t^{1-\gamma}}{1-\gamma}+\frac{2t^{2-\gamma}}{2-\gamma}+
 \frac{\Gamma(1-\gamma)\Gamma(1+\nu_1)t^{1+\nu_1-\gamma}}{\Gamma(2-\gamma+\nu_1)}
 +
 \frac{\Gamma(2-\gamma)\Gamma(1+\nu_1)t^{2-\gamma+\nu_1}}{\Gamma(3-\gamma+\nu_1)}.
\end{align*}
The outcome of numerical computations in \cite[Table 7.1]{HPV} for
$\nu_1=0.1+0.i,$ $i=1,2,...,8,$ demonstrates that $\nu_{1,a}<\nu_1$.
\begin{table}
  \begin{center}
 \caption{The quantities $\nu_1$ and $\nu_{1,a}$ in \cite[Table 7.1]{HPV}}
   \label{tab:e2.1}
  \begin{tabular}{c|c|c|c|c|c|c|c|c}
           \hline
       \!$\nu_1$   & 0.1 & 0.2 & 0.3 & 0.4 & 0.5 & 0.6 & 0.7 & 0.8 \\
      \!$\nu_{1,a}$ & 0.0867 & 0.1877 & 0.2920 & 0.3890 & 0.4894 & 0.5904 & 0.6881 & 0.7878\\
 \hline
 \end{tabular}
  \end{center}
\end{table}
In conclusion, Lemma \ref{l7.1} along with Remark \ref{r7.1} and
initial-value problem \eqref{7.4} give the sufficient conditions
providing \eqref{2.15} and the corresponding example demonstrating
the fulfillment all  the assumptions in Theorems
\ref{t2.4}-\ref{t2.5}.


\section{Algorithm of the Simultaneous Reconstruction of $(\nu_1,\nu_{i^*})$ or $(\nu_1,\gamma)$}
\label{s8}

\subsection{Description of the reconstruction algorithm.}
\label{s8.1}

Suppose that we have integral measurements of the solution
$u(x,t)$ to \eqref{i.2}-\eqref{i.3} at discrete time moments
$t_{k},$ $k=1,2,\ldots,K$, $0<t_{1}<t_{2}<\ldots<t_{K}\leq t^{*}$.
We also assume the presence of noise $\{\delta_k\}_{k=1}^K$ affecting the observations, so that, in
practice, we thus have
\[
\psi_{\delta,k}=\int_{\Omega}u(x,t_{k})dx+\delta_{k},\quad
k=1,2,\ldots,K.
\]
The initial condition implies that
\[
\int_{\Omega}u(x,0)dx=\int_{\Omega}u_{0}(x)dx=\psi_{0}.
\]
Formulas \eqref{i.13}, \eqref{i.14}, \eqref{i.15} for recovering unknown parameters involve limits with respect to the continuous variable $t$. Thus, to apply these formulas to discrete noisy measurements, a recovery algorithm must first approximately reconstruct the function $\psi(t)=\int_{\Omega}u(x,t)dx$ from the values $\psi_{\delta,k},$ $k=0,1,\ldots,K$, where we also set
\[
\psi_{\delta,0} \equiv \psi_{0} \equiv \int_{\Omega}u_{0}(x)dx.
\]
Note that in \eqref{i.14} and \eqref{i.15}, the functions $\mathcal{F}_\nu(t)$ and $\mathcal{F}_\gamma(t)$ depend on $\psi(t)$ and its fractional derivatives, which must also be reconstructed from~$\{\psi_{\delta,k}\}$.

As suggested by the small-time asymptotics of $\psi(t)$ (see
comments in \cite{HPV, HPSV}), we require that the recovered
function be (at least) square integrable on $(0,t_{K}),$ $t_{K}\leq
t^{*}$ with an unbounded weight $\varrho(t)=t^{-\mathrm{a}},$
$\mathrm{a}\in(0,1)$, $\psi(t)\in L_{t^{-\mathrm{a}}}^{2}(0,t_{K})$.
Minimizing a penalized least-squares functional (within the Tikhonov
regularization framework \cite{LP,IJ,TG}), we construct a function
approximating $\psi(t)$ from the noisy
data~$\{\psi_{\delta,k}\}_{k=0}^{K}$:
\begin{equation}\label{8.1*}
\sum_{k=0}^{K}[\psi(t_{k})-\psi_{\delta,k}]^{2}+ \sigma
\|\psi\|^{2}_{L_{t^{-\mathrm{a}}}^{2}(0,t_{K})}\longrightarrow \min,
\end{equation}
where $\sigma$ is a regularization parameter. For the practical implementation of parameters reconstruction,
we look for an approximate minimizer of \eqref{8.1*} in the following finite-dimensional form:
\begin{equation}\label{8.2}
\psi_{\delta}(\sigma,t)=\sum_{j=1}^{\mathfrak{I}}q_{j}t^{\beta_{j}}
+\sum_{j=\mathfrak{I}+1}^{\mathfrak{P}}q_{j}P_{j-\mathfrak{I}-1}^{(0,-\mathrm{a})}(t/t_{K}).
\end{equation}
Here the shifted Jacobi polynomials are conventionally defined as~\cite{KPSV1, KPSV2, HPV, HPSV}
\[
P_{m}^{(0,-\mathrm{a})}(t/t_{K})=\sum_{i=0}^{m}\left(\begin{array}{c}
    m\\
    i
\end{array}\right)
\left(\begin{array}{c}
    m-\mathrm{a}\\
   m-i
\end{array}\right)\!(t/t_{K}-1)^{m-i}(t/t_{K})^{i}\quad\text{with}\quad
t\in(0,t_{K}) ,
\]
and they form an orthogonal system in
$L_{t^{-\mathrm{a}}}^{2}(0,t_{K})$. Power functions $t^{\beta_{j}}$
($j=1,2,\ldots,\mathfrak{I}$) help capture the small-time
asymptotics of the exact problem solution \cite{KPSV1, KPSV2, HPV,
HPSV} (so that $\beta_{1}<\beta_{2}<\ldots<\beta_{\mathfrak{I}}$ are
the initial fractional order guesses, if any). Using
\cite[Corollary~2.7]{Heit2009} we can also rewrite the above
polynomials in the following concise form which appears to be more
convenient when calculating the Caputo derivatives and convolutions
of the approximating function:
\begin{equation}
\label{factorized_Jacobi_pol} P_{m}^{(0,-\mathrm{a})}(t/t_{K}) =
\sum_{i=0}^m (-1)^{m-i} \left(\begin{array}{c}
    m\\
    i
\end{array}\right) \left(\begin{array}{c}
    m-\mathrm{a}+i\\
    m
\end{array}\right) (t/t_K)^i .
\end{equation}


\noindent The unknown coefficients $q_{j}$ in \eqref{8.2} are identified from the following system of linear algebraic equations, which in matrix form reads:
\[
(\mathbb{E}^{T}\mathbb{E}+\sigma\mathbb{H})\mathbf{q}=\mathbb{E}^{T}\bar{\psi_{\varepsilon}},
\]
where we set
\begin{align*}
\mathbf{q}&=(q_{1},\ldots,q_{\mathfrak{P}}),\quad
\bar{\psi}_{\delta}=(\psi_{\delta,0},\psi_{\delta,1},\ldots,\psi_{\delta,K})^{T},\\
\mathbb{E}&=\{E_{ij}\}^{K,\quad\mathfrak{P}}_{i=0,j=1},\quad
E_{ij}=e_{j}(t_{i}),\\
\mathbb{H}&=\{H_{l,m}\}_{l,m=1}^{\mathfrak{P}},\quad
H_{l,m}=\int_{0}^{t_{K}}t^{-\mathrm{a}}e_{l}(t)e_{m}(t)dt,\\
e_{l}(t)&=
\begin{cases}
t^{\beta_{l}},\quad l=1,2,\ldots,\mathfrak{I},\\
P_{l-\mathfrak{I}-1}^{(0,-\mathrm{a})}(t/t_{K}),\quad
l=1+\mathfrak{I},\ldots,\mathfrak{P}.
\end{cases}
\end{align*}
Note that, provided the regularization parameter $\sigma$ is
found and the approximate recovery of $\psi(t)$ in the form $\psi_{\delta}(\sigma,t)$ is performed, we still need to compute the limits in the reconstruction formulas \eqref{i.13}, \eqref{i.14}, \eqref{i.15}. In general, the computation of limits is an ill-posed problem from a numerical point of view \cite{LP}, and thus requires regularization as well. As an approximate value of the $\nu_1$-limit, one could use the quantity
\begin{equation}\label{8.1}
\nu_{1,\delta}(\sigma,\bar{t})=
\begin{cases}
\frac{\ln|\psi_{\delta}(\sigma,\bar{t})-\psi_{0}|}{\ln
\bar{t}}\qquad\qquad\qquad\text{in the case of the I type FDO},\\
\\
\frac{\ln|\rho_{1}(\bar{t})\psi_{\delta}(\sigma,\bar{t})-\rho_{1}(0)\psi_{0}|}{\ln
\bar{t}}\qquad\text{ in the case of the II type FDO},
\end{cases}
\end{equation}
computed at a point $t=\bar{t}$ sufficiently close to zero, which can also be treated as a regularization parameter. As for recovering $\nu_{i^*}$ in FIP or $\gamma$ in SIP, recalling the magnitude of $\bar{\nu}_{1,\delta} = \nu_{1,\delta}(\sigma,\bar{t})$ and taking into account \eqref{i.14} and \eqref{i.15} we arrive~at
\begin{equation}\label{8.5}
\bar{\nu}_{i^*,\delta}=\nu_{i^*,\delta}({\sigma},\bar{t})=\bar{\nu}_{1,\delta}
-\log_{\lambda}\bigg|\frac{\mathcal{F}_{\nu,\delta}({\sigma},\bar{t}\lambda)}{\mathcal{F}_{\nu,\delta}({\sigma},\bar{t})}\bigg|, \qquad\qquad
\bar{\gamma}_\delta = \gamma_\delta({\sigma},\bar{t}) = 1-\log_{\mu}\bigg|\frac{\mathcal{F}_{\gamma,\delta}({\sigma},\mu\bar{t})}{\mathcal{F}_{\gamma,\delta}({\sigma},\bar{t})}\bigg|,
\end{equation}
where the parameters $\lambda,\mu\in(0,1)$ are selected by the user, and the functions $\mathcal{F}_{\nu,\delta}(\sigma,t)$ and $\mathcal{F}_{\gamma,\delta}(\sigma,t)$ are naturally calculated by substituting the approximations $\psi_{\delta}(\sigma,t)$ and $\bar{\nu}_{1,\delta}$ for the unknown continuous-argument $\psi(t)$ and order $\nu_1$ in the original definitions of functions $\mathcal{F}_{\nu}(t)$ and $\mathcal{F}_{\gamma}(t)$ (see \eqref{i.12},~\eqref{i.12*}). Note that this process requires calculating the fractional derivatives of the basis functions in \eqref{8.2}, particularly the Jacobi polynomials, where the factorized representation \eqref{factorized_Jacobi_pol} appears to be more suitable than their original definition. Thus, the pair $(\bar{\nu}_{1,\delta},\bar{\nu}_{i^*,\delta})$ (or $(\bar{\nu}_{1,\delta},\bar{\gamma}_\delta)$) computed via \eqref{8.1}-\eqref{8.5} could be considered the outcome of the proposed approximation scheme (at fixed arguments $\sigma$,~$\bar t$).


In practice, however, to calculate the regularized approximations $\nu_{1,\delta}(\sigma,\bar{t})$ and $\nu_{i^*,\delta}(\sigma,\bar{t})$ (or ${\gamma}_\delta(\sigma,\bar{t})$) of the true orders $\nu_{1}$ and $\nu_{i^*}$ (or $\gamma$) via the above scheme \eqref{8.1}-\eqref{8.5}, we still have to properly choose two regularization parameters: $\sigma$ and $\bar{t}$. Since the noise perturbation amplitudes $\delta_{k}$ are typically unknown, it is reasonable to rely on the so-called noise-level-free
regularization parameter choice rules \cite{KPSV1,KPSV2,HPV,HPSV,PSV2}.  First, the quasi-optimality criterion \cite{TG}, a heuristically motivated a posteriori rule, is one of the simplest and most efficient approaches among such strategies. Its application to the selection of multiple parameters has been discussed in \cite{FNP, LP}, and its effectiveness has been advocated in \cite{BK}.
Here, to utilize this criterion, we introduce two geometric sequences of regularization parameters values,
\begin{equation}
\label{sequences_quasiopt_0}
\sigma = \sigma_{i}=\sigma_{1}\xi_{1}^{i-1},\quad
i=1,2,\ldots,K_{1},\quad\text{and}\quad
\bar{t}=\bar{t}_{j}=\bar{t}_{1}\xi_{2}^{j-1},\quad
j=1,2,\ldots,K_{2},
\end{equation}
with (user-defined) values $\sigma_1$ and $\bar{t}_{1},$ and
$\xi_{1},\xi_{2}\in(0,1)$. The magnitudes of $\nu_{1,\delta}(\sigma_{i},\bar{t}_{j})$ and $\nu_{i^*,\delta}(\sigma_{i},\bar{t}_{j})$ (or $\gamma_\delta({\sigma},\bar{t})$) have to be calculated for such indices $i$ and~$j$. After that for each $\bar{t}_{j}$ we seek for $\sigma_{i_{j}}\in\{\sigma_{i}\}_{i=1}^{K_{1}}$ such that
\begin{subequations}
\label{reconstrunction_alg_}
\begin{gather}
\|\Vec\nu_{\delta}(\sigma_{i_{j}},\bar{t}_{j})-\Vec\nu_{\delta}(\sigma_{i_{j}-1},\bar{t}_{j})\|^\prime =
\min\{\|\Vec\nu_{\delta}(\sigma_{i},\bar{t}_{j})-\Vec\nu_{\delta}(\sigma_{i-1},\bar{t}_{j})\|^\prime,\quad
i=2, 3,\ldots,K_{1}\},\label{reconstrunction_alg_step1_}
\intertext{where $\Vec\nu_\delta$ is the two-component vector defined by $\Vec\nu_\delta \equiv (\nu_{1,\delta},\nu_{i^*,\delta})$ (or $\Vec\nu_\delta \equiv (\nu_{1,\delta},\gamma_{\delta})$), where its 1st component (i.e.~$\nu_{1,\delta}$) is computed via \eqref{8.1}, whereas the 2nd one is then determined by employing relation \eqref{8.5} and the $\nu_{1,\delta}$ value just calculated. One thus must thus strive to calculate the 1st component as accurately as possible, since the 2nd one is also determined through it in this scheme; hence, to make the algorithm more selective with respect to the 1st component of vector $\Vec\nu_\delta$ we use the weighted Euclidean norm, $$\|\Vec\nu_\delta\|^\prime = \bigl((\Upsilon \nu_{1,\delta})^2 + (\nu_{i^*,\delta})^2\bigr)^{1/2}$$ (or, respectively, $$\|\Vec\nu_\delta\|^\prime = \bigl((\Upsilon \nu_{1,\delta})^2 + (\gamma_{\delta})^2\bigr)^{1/2}$$ if identifying $(\nu_1,\gamma)$) with some weight $\Upsilon$, to make the 1st component more dominant (by default in our numerical examples below we have used $\Upsilon=10$ to reinforce the 1st component with (at least) one decimal digit). Next, $\bar{t}_{j_{0}}$ is selected from $\{\bar{t}_{j}\}_{j=1}^{K_{2}}$ such that}
\|\Vec\nu_{\delta}(\sigma_{i_{j_0}},\bar{t}_{j_0})-\Vec\nu_{\delta}(\sigma_{i_{j_0-1}},\bar{t}_{j_0-1})\|^\prime
=
\min\{\|\Vec\nu_{\delta}(\sigma_{i_{j}},\bar{t}_{j})-\Vec\nu_{\delta}(\sigma_{i_{j-1}},\bar{t}_{j-1})\|^\prime,\quad
j=2,3,\ldots,K_{2}\}.\label{reconstrunction_alg_step2_}
\end{gather}
\end{subequations}
At last, $\Vec\nu_{\delta}(\sigma_{i_{j_{0}}},\bar{t}_{j_{0}})$ (with a bit abuse of symbols henceforth simply denoted componentwisely by $(\bar{\nu}_{1,\delta}, \bar{\nu}_{i^*,\delta})$ or $(\bar{\nu}_{1,\delta}, \bar{\gamma}_{\delta})$ for brevity) is chosen as the final output of the reconstruction algorithm. In the next subsection, we demonstrate its performance by numerical examples.


\subsection{Numerical experiments}
\label{s8.2}

Here we discuss \eqref{i.2}-\eqref{i.4} stated  in $\Omega_{T}$ with the terminal time $T=1$ and 2D domain $\Omega = (0,1)\times(0,1)$. Moreover, we focus on the partial case of \eqref{i.2}-\eqref{i.4} reading~as
\begin{equation}\label{9.1}
\begin{cases} \mathbf{D}_{t}u-\Delta u-a_{0}(t)u-\mathcal{K}*\Delta
u-\mathcal{K}*b_{0}u=\sum_{i=1}^{3}g_{i}(x,t) \equiv g(x,t) \quad\text{in}\quad\Omega_{T},\\
u(x,0)=u_{0}(x)\quad\text{in}\quad\bar{\Omega},\qquad \frac{\partial
u}{\partial\mathbf{N}}=0\quad\text{on}\quad\partial\Omega_{T}.
\end{cases}
\end{equation}
The noisy observation \eqref{i.4} is simulated via
\begin{equation}\label{9.2}
\psi_{\delta,
k}=\int_{\Omega}u(x,t_{k})dx+\delta\mathfrak{G}(t_{k}),\quad
k=1,2,\ldots,K
\end{equation}
with a numerical $\delta$ specified below and the noisy data $\mathfrak{G}$ having  the form
\[
\mathfrak{G}(t)=\begin{cases}
t|\ln\, t| &\text{First type noise (\textbf{FTN}) case},\\
t^{\nu_{1}} & \text{Second type noise (\textbf{STN}) case},\\
t^{\nu_{1}}|\ln\, t| &\text{Third type noise (\textbf{TTN}) case}.
\end{cases}
\]

Below we demonstrate the performance of the reconstruction algorithm
discussed in Section~\ref{s8.1} on several examples. In all test
examples the data $u(x,t_{k})$ is generated via explicit solutions
to the respective initial-boundary value problems (naturally, the
corresponding data can also be acquired numerically using
finite-difference schemes such as e.g.~those considered
in~\cite{PSV1, SV, PSV2}, or by leveraging finite-element
approaches~\cite{JinLazZ, SVjcp, SVcsa}).

\begin{example}\label{e.2}
We consider \eqref{9.1} with
\[
M=3,\, \nu_{1}=\nu,\, \nu_{2}=\frac{\nu}{2},\,
\nu_{3}=\frac{\nu}{3}, \,\rho_{1}(t)=1/2,\,
\rho_{2}(t)=-\frac{1}{4},\, \rho_{3}(t)=\frac{1+t^{2}}{4},\,
a_{0}(t)=2,\, b_{0}(t)=0,
\]
and
\begin{flalign*}
\mathcal{K}(t) & =t^{-\gamma}\quad\text{with}\quad \gamma\in(0,1),
\quad u_{0}(x,y)=x^{2}(1-x)^{2} + y^{2}(1-y)^{2}, &\\
g_{1}(x,y,t)&=\bigl(x^{2}(1-x)^{2}+y^{2}(1-y)^{2}\bigr)\biggl[\frac{15}{2}\Gamma(1+\nu)-\frac{15}{4}\frac{\Gamma(1+\nu)}{\Gamma(1+\frac{\nu}{2})}t^{\frac{\nu}{2}}
+\frac{t^{2-\frac{\nu}{3}}}{2\Gamma(3-\frac{\nu}{3})} \\
&\quad +\frac{15\Gamma(1+\nu)}{4\Gamma(1+\frac{2\nu}{3})}t^{\frac{2\nu}{3}}
+\frac{15\Gamma(3+\nu)}{4\Gamma(3+\frac{2\nu}{3})}t^{2+\frac{2\nu}{3}}
 \biggr] ,\\
g_{2}(x,y,t)&=-2[1+15 t^{\nu}][2-6 x+7 x^{2}-2 x^{3}+x^{4}-6 y+7 y^{2}-2 y^{3}+y^{4}] ,\\
g_{3}(x,y,t)&=-4(1-3x+3x^{2}-3y+3y^{2})\Big[\frac{t^{1-\gamma}}{1-\gamma}+15 t^{1-\gamma+\nu}\frac{\Gamma(1-\gamma)\Gamma(1+\nu)}{\Gamma(2+\nu-\gamma)}\Big].
\end{flalign*}
\end{example}
\noindent
That is, one has FDO $\mathbf{D}_{t}u=\frac{1}{2}\mathbf{D}_{t}^{\nu}u-\frac{1}{4}\mathbf{D}_{t}^{\nu/2}u + \mathbf D_t^{\nu/3}\bigl(\frac{1+t^2}{4} u\bigr)$. The direct calculations gives the explicit form of the solution,
\[
u(x,y,t)=\bigl(x^{2}(1-x)^{2}+y^{2}(1-y)^{2}\bigr)[1+ 15 t^{\nu}]
\]
to this initial-boundary value problem and, besides, we have
\[
\int_{\Omega}u_{0}(x,y) dx dy =\frac{1}{15}.\] In this numerical test we test FIP
focusing on numerical recovering $\nu_{1}=\nu$ and $\nu_{3}=\nu/3$ for given values of $\nu$
(naturally, treating $\nu_1$ and $\nu_3$ as unknown when running the reconstruction algorithm).
Note that the condition $\c_{\nu,0}\neq0$ of Theorem~\ref{t2.1} is satisfied ($\c_{\nu,0} = \Gamma(1+\nu)/2\neq0$).
 The corresponding calculation outcomes of the proposed reconstruction algorithm (see Section~\ref{s8.1}) are listed in
 Table~\ref{tab:Ex2_Nn_2D}. Here we have used the low-order regression model with Jacobi polynomials $P_{m}^{(0,-\mathrm{a})}$
  with $0\le m\le 5$, $\mathrm{a}=0.99$, three power
   functions $\{t^{0.25}, t^{0.5}, t^{0.75}\}$ as initial guesses to
   facilitate the small-time asymptotics capturing (i.e.~$\beta_1=0.25$, $\beta_2=0.5$, $\beta_3=0.75$),
   regularizing sequences $\sigma = \{2^{1-i}\}_{i=1}^{K_1}$ and $\bar t = \{2^{1-j} t_K \}_{j=1}^{K_2}$ with $K_1=50$ and $K_2=20$
    (see \eqref{sequences_quasiopt_0}), observation time moments $\{t_k\} = \{k\tau\}_{k=1}^K$ with $K=20$
    and $\tau=10^{-2}$, $\lambda=0.99$ (see \eqref{8.5}). It is clearly observed that decreasing the noise
    level and softening the noise type (from the severe 3rd type to the lighter 2nd and 1st types) improves the calculation results,
    as expected (though the procedure recovers the major fraction order $\nu_1$ very accurately in all cases).
     Note also that for small $\nu$ (namely, $\nu=0.1, 0.3$) and stronger ($\delta=0.01$) 3rd type noise, the recovered
      value $\bar{\nu}_{1,\delta}$ is underestimated, whereas $\bar{\nu}_{3,\delta}$ is significantly overestimated;
      as in the case of a simpler single-parameter recovery described in \cite{PSV2}, this can be fixed by more
      thoroughly tuning the approximation model \eqref{8.2} and the selection algorithm
       (relations \eqref{sequences_quasiopt_0}-\eqref{reconstrunction_alg_} in the present study),
       which may involve, as demonstrated in \cite{PSV2}, making more pertinent initial guesses $\beta_j$ regarding
       the magnitude of fractional orders, increasing the observable dataset and the model dimensionality and/or allowing more
       steps in the selection algorithm -- for instance, taking smaller values of guesses (reasonable
        for recovering small values of fractional orders), e.g.~$\beta_1=0.02$, $\beta_2=0.09$, $\beta_3=0.15$,
        one obtains noticeably more accurate outcomes $(\bar{\nu}_{1,\delta}, \bar{\nu}_{3,\delta}) = (0.0907, 0.0462)$
         and $(\bar{\nu}_{1,\delta}, \bar{\nu}_{3,\delta}) = (0.1902, 0.0903)$ at $\nu=0.1, 0.3$ in this case.
\begin{table}
  \footnotesize
  \begin{center}
    \caption{The reconstructed $(\nu_1, \nu_3)$ values in Example~\ref{e.2}.}
    \label{tab:Ex2_Nn_2D}
 \begin{tabular}{@{}c|c|c|c|c|c|c|c|c|c|c|c|c@{}}
      \hline
      \multicolumn{6}{c|}{$\delta=0.01 \ $} & \multicolumn{6}{c|}{$\ \delta=0.001$} & \\
      \hline
      \multicolumn{2}{c|}{\textbf{FTN}} & \multicolumn{2}{c|}{\textbf{STN}} & \multicolumn{2}{c|}{\textbf{TTN}} & \multicolumn{2}{c|}{\textbf{FTN}} & \multicolumn{2}{c|}{\textbf{STN}} & \multicolumn{2}{c|}{\textbf{TTN}} & \\
      \hline
      $\bar{\nu}_{1,\delta}$ & $\bar{\nu}_{3,\delta}$ & $\bar{\nu}_{1,\delta}$ & $\bar{\nu}_{3,\delta}$ & $\bar{\nu}_{1,\delta}$ & $\bar{\nu}_{3,\delta}$ & $\bar{\nu}_{1,\delta}$ & $\bar{\nu}_{3,\delta}$ & $\bar{\nu}_{1,\delta}$ & $\bar{\nu}_{3,\delta}$ & $\bar{\nu}_{1,\delta}$ & $\bar{\nu}_{3,\delta}$ & $\nu$ \\
      \hline
0.0998 & 0.0279 & 0.0977 & 0.0320 & 0.0902 & 0.0792 & 0.1000 & 0.0305 & 0.0998 & 0.0309 & 0.0990 & 0.0367 & 0.1 \\
0.1999 & 0.0608 & 0.1977 & 0.0662 & 0.1902 & 0.1027 & 0.2000 & 0.0650 & 0.1998 & 0.0654 & 0.1990 & 0.0696 & 0.2 \\
0.2996 & 0.1091 & 0.2980 & 0.1106 & 0.2902 & 0.1210 & 0.3000 & 0.1004 & 0.2998 & 0.1007 & 0.2990 & 0.1037 & 0.3 \\
0.3995 & 0.1513 & 0.3981 & 0.1538 & 0.3902 & 0.1526 & 0.3999 & 0.1346 & 0.3998 & 0.1349 & 0.3990 & 0.1369 & 0.4 \\
0.4998 & 0.1814 & 0.4985 & 0.1807 & 0.4903 & 0.1802 & 0.5000 & 0.1681 & 0.4998 & 0.1681 & 0.4990 & 0.1684 & 0.5 \\
0.5987 & 0.2074 & 0.5980 & 0.2039 & 0.5902 & 0.2159 & 0.5998 & 0.1982 & 0.5998 & 0.1981 & 0.5990 & 0.1992 & 0.6 \\
0.6983 & 0.2476 & 0.6980 & 0.2444 & 0.6902 & 0.2519 & 0.6997 & 0.2325 & 0.6998 & 0.2323 & 0.6990 & 0.2334 & 0.7 \\
0.7953 & 0.2756 & 0.7977 & 0.2908 & 0.7902 & 0.2906 & 0.7996 & 0.2708 & 0.7998 & 0.2706 & 0.7990 & 0.2722 & 0.8 \\
0.8932 & 0.3192 & 0.8973 & 0.3182 & 0.8902 & 0.3353 & 0.8993 & 0.2993 & 0.8997 & 0.2990 & 0.8990 & 0.3016 & 0.9 \\
 \end{tabular}
  \end{center}
\end{table}

\begin{example}
\label{e.2.SIP}
To test the proposed algorithm for SIP in recovering the pair $(\nu_1,\gamma)$ let us modify the previous
Example~\ref{e.2} by fixing $\gamma=0.9$ and setting $b_0(t)=15$ (note that in the previous example with
zero $b_0$, the essential condition $\c_{1,0}\neq 0$ of the solvability Theorem~\ref{t2.2} for SIP was
 not satisfied). All other coefficients remain the same, except for the right-hand side $g(x,y,t)$, to which the additional term
$$
-t^{1-\gamma}\Bigl(\frac{1}{1-\gamma}+15 t^\nu\frac{\Gamma(1-\gamma)\Gamma(1+\nu)}{\Gamma(2-\gamma+\nu)}\Bigr)
$$
is added. The previous settings of the (low-order) reconstruction
model and algorithm are also retained; $\mu=10^{-2}$. The
corresponding outcomes of the proposed reconstruction algorithm (see
Section~\ref{s8.1}) are listed in Table~\ref{tab:Ex2_Nn_2D_SIP}. It
can be seen that for small $\nu$ (such as e.g.~$\nu\approx0.1$) the
recovered $\bar{\gamma}_{\delta}$ value noticeably underestimates
(by $\approx0.1$) the true $\gamma=0.9$; this, however, as in the
previous example, can be remedied by finer tuning of the model and
the selection algorithm.
\begin{table}
  \footnotesize
  \begin{center}
    \caption{The reconstructed $(\nu_1,\gamma)$ values in Example~\ref{e.2.SIP}.}
    \label{tab:Ex2_Nn_2D_SIP}
 \begin{tabular}{@{}c|c|c|c|c|c|c|c|c|c|c|c|c@{}}
      \hline
      \multicolumn{6}{c|}{$\delta=0.01 \ $} & \multicolumn{6}{c|}{$\ \delta=0.001$} & \\
      \hline
      \multicolumn{2}{c|}{\textbf{FTN}} & \multicolumn{2}{c|}{\textbf{STN}} & \multicolumn{2}{c|}{\textbf{TTN}} & \multicolumn{2}{c|}{\textbf{FTN}} & \multicolumn{2}{c|}{\textbf{STN}} & \multicolumn{2}{c|}{\textbf{TTN}} & \\
      \hline
      $\bar{\nu}_{1,\delta}$ & $\bar{\gamma}_{\delta}$ & $\bar{\nu}_{1,\delta}$ & $\bar{\gamma}_{\delta}$ & $\bar{\nu}_{1,\delta}$ & $\bar{\gamma}_{\delta}$ & $\bar{\nu}_{1,\delta}$ & $\bar{\gamma}_{\delta}$ & $\bar{\nu}_{1,\delta}$ & $\bar{\gamma}_{\delta}$ & $\bar{\nu}_{1,\delta}$ & $\bar{\gamma}_{\delta}$ & $\nu$ \\
      \hline
0.0998 & 0.8121 & 0.0977 & 0.8120 & 0.0901 & 0.8096 & 0.1000 & 0.8121 & 0.0998 & 0.8121 & 0.0990 & 0.8098 & 0.1 \\
0.3962 & 0.8525 & 0.3952 & 0.8523 & 0.3866 & 0.8524 & 0.3963 & 0.8525 & 0.3962 & 0.8524 & 0.3953 & 0.8525 & 0.4 \\
0.6011 & 0.8684 & 0.6004 & 0.8680 & 0.5919 & 0.8688 & 0.6016 & 0.8682 & 0.6015 & 0.8682 & 0.6006 & 0.8683 & 0.6 \\
0.8940 & 0.8972 & 0.8977 & 0.8968 & 0.8895 & 0.8971 & 0.8987 & 0.8970 & 0.8991 & 0.8970 & 0.8982 & 0.8970 & 0.9 \\
 \end{tabular}
  \end{center}
\end{table}

\end{example}



\end{document}